\documentclass{article}

\usepackage{amssymb}
\usepackage{amsthm}
\usepackage{amsmath}
\usepackage{amsthm}
\usepackage{graphicx}
\usepackage{subfig}

\usepackage{a4wide}
\usepackage[figuresright]{rotating}

\usepackage{url}

\usepackage{cite}

\newtheorem{theorem}{Theorem}[section]
\newtheorem*{lemma}{Lemma}

\theoremstyle{remark}
\newtheorem{remark}{Remark}[section]

\begin{document}

\title{Cost effectiveness analysis of optimal control measures\\ for tuberculosis\thanks{This 
is a preprint of a paper whose final and definite form is published in the 
\emph{Bulletin of Mathematical Biology}, ISSN 0092-8240, {\tt http://dx.doi.org/10.1007/s11538-014-0028-6}. 
Submitted 04/July/2014; Revised 07/Sept/2014; Accepted 11/Sept/2014.}}


\author{Paula Rodrigues$^1$\\
{\tt pcpr@fct.unl.pt}
\and
Cristiana J. Silva$^2$\\
{\tt cjoaosilva@ua.pt}
\and
Delfim F. M. Torres$^2$\\
{\tt delfim@ua.pt}}

\date{$^1$Department of Mathematics and Center of Mathematics and Applications,\\
New University of Lisbon, 2829--516, Caparica, Portugal\\[0.3cm]
$^2$Center for Research and Development in Mathematics and Applications (CIDMA),
Department of Mathematics, University of Aveiro,\\ 3810--193 Aveiro, Portugal}

\maketitle


\begin{abstract}
We propose and analyse an optimal control problem where the control system
is a mathematical model for tuberculosis that considers reinfection.
The control functions represent the fraction of early latent and persistent
latent individuals that are treated. Our aim is to study how these control
measures should be implemented, for a certain time period, in order to reduce
the number of active infected individuals, while minimizing the interventions
implementation costs. The optimal intervention is compared along different
epidemiological scenarios, by varying the transmission coefficient.
The impact of variation of the risk of reinfection, as a result of acquired
immunity to a previous infection for treated individuals on the optimal controls
and associated solutions, is analysed. A cost-effectiveness analysis is done,
to compare the application of each one of the control measures,
separately or in combination.

\bigskip

\noindent {\bf Keywords:} tuberculosis; optimal control;
post-exposure interventions; efficacy function; cost effort.

\smallskip

\noindent {\bf Mathematics Subject Classification 2010:} 92D30; 49M05.
\end{abstract}


\section{Introduction}
\label{sec:Int}

Tuberculosis (TB) detection and treatment saved 22 million of lives, between 1995 and 2012,
following the 2013 report of the World Health Organization (WHO) \cite{WHO_2013}. However,
in 2012, there were 8.6 million of new TB cases and 1.3 million of TB deaths \cite{WHO_2013}.
TB prevention, diagnosis and treatment, requires adequate funding, sustained over many years,
which represents a worldwide scale challenge.

Mathematical dynamic models are an important tool in analyzing the spread and control
of infectious diseases. Many TB mathematical models have been developed --- see, e.g.,
\cite{Blower_Small_Hopewell_1996,Castillo_Feng_1998,Cohen_Murray_2004,Dye_et_all_1998,Gomes_et_all_2004,Vynnycky_Fine_1997}
and the references cited therein. The main differences of the models proposed in
\cite{Castillo_Feng_1997,Cohen_Colijn_Finklea_Murray_2006,Cohen_Murray_2004,Dye_et_all_1998,Gomes_et_all_2004,Gomes_etall_2007,%
Rie_et_all_2005,Rodrigues_et_all_2007,Vynnycky_Fine_1997,Warren_et_all_2004} are the way they represent reinfection,
since there is no consensus on wether a previous infection gives or not protection. The way recently infected individuals
progress to active disease is not the same in all models: they can be ``fast progressors'' or ``slow progressors''.
In some models, it is assumed that only 5 to 10\% of the infected individuals are fast progressors. The remaining models
consider that individuals are able to contain the infection asymptomatic and non infectiously (latent individuals),
having a much lower probability of developing active disease by endogenous reactivation. More recent models also assume exogenous
reinfection of latent and treated individuals, based on the fact that infection and/or disease do not confer full protection
\cite{Verver_etall_2005}. This assumption has an important impact on the efficacy of interventions
\cite{Cohen_Murray_2004,Gomes_et_all_2004,Rodrigues_et_all_2007,MyID:265,MyID:230,MyID:271,Vynnycky_Fine_1997}.
In this paper, we consider a TB mathematical model from \cite{Gomes_etall_2007},
where exogenous reinfection is considered.

Without treatment, TB mortality rates are hight \cite{WHO_2013}. Different interventions are available
for TB prevention and treatment: vaccination to prevent infection; treatment to cure active TB;
treatment of latent TB to prevent endogenous reactivation. In this work, we study the implementation
of two post-exposure interventions that are not widely used: treatment of early latent individuals
with anti-TB drugs (e.g., treatment of recent contacts of index cases) and prophylactic treatment/vaccination
of the persistent latent individuals. We propose an optimal control problem that consists in analyzing
how these two control measures should be implemented, for a certain time period, in order to reduce
the number of active infected individuals, while controlling the interventions implementation costs.

Optimal control is a branch of mathematics developed to find optimal ways to control a dynamic system
\cite{Cesari_1983,Fleming_Rishel_1975,Pontryagin_et_all_1962}. Other authors applied optimal control
theory to TB models (see, e.g., \cite{SLenhart_2002,TB:Cameroon,MyID:230}). This approach allows the study of the most cost-effective
intervention design by generating an implementation design that minimizes an objective function. The intensity
of interventions can be relaxed along time, which is not the case considered in most models,
for which interventions are modeled by constant rates \cite{Gomes_etall_2007}.

The paper is organized as follows. In Section~\ref{sec:model} we present the mathematical
model for TB that will be study in this paper. Two control functions $u_1$ and $u_2$ are
then added to the original model from \cite{Gomes_etall_2007}. Section~\ref{sec:oc:problem}
is dedicated to the formulation of the optimal control problem. We prove the existence of an
unique solution and derive the expression for the optimal controls according
to the Pontryagin maximum principle \cite{Pontryagin_et_all_1962}.
Section~\ref{sec:numericalresults} has four subsections dedicated to a numerical
and cost-effectiveness analysis of the optimal control problem. We start by illustrating
the problem solutions for a particular case (Section~\ref{subsec:example}).
We then introduce some summary measures in Section~\ref{subsec:smeasures}
to describe how the results change when varying transmission intensity
(Section~\ref{subsec:beta}) and protection against reinfection (Section~\ref{subsec:sigma}).
In Section~\ref{subsec:oc:strategies}, we analyze the cost-effectiveness of three intervention strategies:
applying $u_1$ or $u_2$ separately and applying the two control measures simultaneously.
We end with Section~\ref{sec:discussion} of discussion.


\section{Mathematical model}
\label{sec:model}

Following the model proposed in \cite{Gomes_etall_2007}, population is divided into five categories:
susceptible ($S$); early latent ($L_1$), \textrm{i.e.}, individuals recently infected
(less than two years) but not infectious; infected ($I$), \textrm{i.e.},
individuals who have active TB and are infectious;
persistent latent ($L_2$), \textrm{i.e.}, individuals who were infected and remain latent;
and recovered ($R$), \textrm{i.e.}, individuals who were previously infected and treated.

We assume that at birth all individuals are equally susceptible
and differentiate as they experience infection and respective therapy.
The rate of birth and death, $\mu$, are equal (corresponding to a mean life
time of 70 years \cite{Gomes_etall_2007}) and no disease-related deaths are considered,
keeping the total population, $N$, constant with $N = S(t) + L_1(t) + I(t) + L_2(t) + R(t)$.

Parameter $\delta$ denotes the rate at which individuals leave $L_1$ compartment;
$\phi$ is the proportion of infected individuals progressing directly to the active
disease compartment $I$; $\omega$ and $\omega_R$ are the rates of endogenous reactivation
for persistent latent infections (untreated latent infections) and for treated individuals
(for those who have undergone a therapeutic intervention), respectively.
Parameters $\sigma$ and $\sigma_R$ are factors that reduce the risk of infection,
as a result of acquired immunity to a previous infection, for persistent latent individuals
and for treated patients, respectively. These factors affect the rate of exogenous reinfection.
As in \cite{Gomes_etall_2007}, in our simulations we consider three different cases
for the protection against reinfection conferred by treatment: same protection
as natural infection ($\sigma_R = \sigma$); lower protection than conferred
by infection ($\sigma_R = 2\sigma$); and higher protection than conferred by infection
($\sigma_R =\sigma/2$), see Section~\ref{subsec:sigma}. Parameter $\tau_0$ is the rate
of recovery under standard treatment of active TB, assuming an average duration
of infectiousness of six months. The values of the rates $\delta$, $\phi$, $\omega$,
$\omega_R$, $\sigma$ and $\tau_0$ are taken from \cite{Gomes_etall_2007}
and the references cited therein (see Table~\ref{parameters} for the values of the parameters).

Additional to standard treatment of infectious individuals, we consider two post-exposure
interventions targeting different sub-populations: early detection and treatment
of recently infected individuals ($L_1$) and chemotherapy or post-exposure vaccine
of persistent latent individuals ($L_2$). These interventions are applied at rates
$\tau_1$ and $\tau_2$. We consider, without loss of generality, that the rate of recovery
of early latent individuals under post-exposure interventions is equal to the rate of recovery
under treatment of active TB, $\tau_1 = 2 \, yr^{-1}$, and greater than the rate of recovery
of persistent latent individuals under post-exposure interventions, $\tau_2 = 1 \, yr^{-1}$
\cite{Gomes_etall_2007}. Since we are interested in studying these interventions along time,
we add to the original model two control functions, $u_1(\cdot)$ and $u_2(\cdot)$,
which represent the intensity at which these post-exposure interventions are applied at each time step.

The dynamical control system that we propose is given by
\begin{equation}
\label{modelGab_controls}
\begin{cases}
\dot{S}(t) = \mu N - \frac{\beta}{N} I(t) S(t) - \mu S(t)\\
\dot{L_1}(t) = \frac{\beta}{N} I(t)\left( S(t) + \sigma L_2(t)
+ \sigma_R R(t)\right) - (\delta + \tau_1 u_1(t) + \mu)L_1(t)\\
\dot{I}(t) = \phi \delta L_1(t) + \omega L_2(t) + \omega_R R(t)
- (\tau_0  + \mu) I(t)\\
\dot{L_2}(t) = (1 - \phi) \delta L_1(t) - \sigma \frac{\beta}{N} I(t) L_2(t)
- (\omega + \tau_2  u_2(t) + \mu)L_2(t)\\
\dot{R}(t) = \tau_0 I(t) +  \tau_1 u_1(t)L_1(t)
+ \tau_2 u_2(t) L_2(t)
- \sigma_R \frac{\beta}{N} I(t) R(t) - \left(\omega_R + \mu\right) R(t) \, .
\end{cases}
\end{equation}

\begin{remark}
The assumption that the total population $N$ is constant, allows to reduce
the control system \eqref{modelGab_controls} from five to four state variables.
We decided to maintain the TB model in form \eqref{modelGab_controls},
using relation $S(t) + L_1(t) + I(t) + L_2(t) + R(t) = N$ as a test
to confirm the numerical results.
\end{remark}

\begin{table}[!htb]
\centering
\begin{tabular}{|l | l | l |}
\hline
{\normalsize{Symbol}} & {\normalsize{Description}}  & {\normalsize{Value}} \\
\hline
{\normalsize{$\beta$}} & {\normalsize{Transmission coefficient}}
& {\normalsize{variable }}\\
{\normalsize{$\mu$}} & {\normalsize{Death and birth rate}}
& {\normalsize{$1/70 \, yr^{-1}$}}\\
{\normalsize{$\delta$}} & {\normalsize{Rate at which individuals leave $L_1$}}
& {\normalsize{$12 \, yr^{-1}$}}\\
{\normalsize{$\phi$}} & {\normalsize{Proportion of individuals going to $I$}}
& {\normalsize{$0.05$}}\\
{\normalsize{$\omega$}} & {\normalsize{Rate of endogenous reactivation for persistent latent infections}}
& {\normalsize{$0.0002 \, yr^{-1}$}}\\
{\normalsize{$\omega_R$}} & {\normalsize{Rate of endogenous reactivation for treated individuals}}
&{\normalsize{$0.00002 \, yr^{-1}$}}\\
{\normalsize{$\sigma$}} & {\normalsize{Factor reducing the risk of infection as a result of acquired}}  & \\
& {\normalsize{immunity to a previous infection for $L_2$}} & {\normalsize{$0.25$}} \\
{\normalsize{$\sigma_R$}} & {\normalsize{Rate of exogenous reinfection of treated patients}}
& {\normalsize{$\sigma; 2\sigma; \sigma/2$}} \\
{\normalsize{$\tau_0$}} & {\normalsize{Rate of recovery under treatment of active TB}}
&  {\normalsize{$2 \, yr^{-1}$}}\\
{\normalsize{$\tau_1$}} & {\normalsize{Rate of recovery under treatment of latent individuals $L1$}}
&  {\normalsize{$2 \, yr^{-1}$}}\\
{\normalsize{$\tau_2$}} & {\normalsize{Rate of recovery under treatment of latent individuals $L2$}}
&  {\normalsize{$1 \, yr^{-1}$}}\\
{\normalsize{$N$}} & {\normalsize{Total population}} & {\normalsize{$30000$}} \\
{\normalsize{$t_f$}} & {\normalsize{Total simulation duration}} & {\normalsize{5 yr}} \\
{\normalsize{$W_0$}} & {\normalsize{Weight constant on active infectious individuals $I(t)$}} & {\normalsize{$50$}}\\
{\normalsize{$W_1$}} & {\normalsize{Weight constant on control $u_1(t)$}}
& {\normalsize{$50$}}\\
{\normalsize{$W_2$}} & {\normalsize{Weight constant on control $u_2(t)$}}
& {\normalsize{$50$}}\\
\hline
\end{tabular}
\caption{Parameter values for the control system \eqref{modelGab_controls}.}
\label{parameters}
\end{table}

It is assumed that the rate of infection of susceptible individuals is proportional
to the number of infectious individuals and the constant of proportionality is $\beta$,
which is the transmission coefficient. The basic reproduction number $R_0$,
for system \eqref{modelGab_controls} in the absence of post-exposure interventions, i.e.,
in the case $u_1 = u_2 = 0$, is proportional to the transmission coefficient $\beta$
(see \cite{Gomes_etall_2007}) and is given by
\begin{equation*}
R_0 = \beta \frac{\delta (\omega + \phi \mu)(\omega_R
+ \mu)}{\mu(\omega_R + \tau_0 + \mu)(\delta + \mu)(\omega + \mu)} \, .
\end{equation*}
The endemic threshold ($ET$) at $R_0 = 1$ indicates the minimal transmission
potential that sustains endemic disease, \textrm{i.e.}, when $R_0 < 1$
the disease will die out and for $R_0 > 1$ the disease may become endemic.
Since our model considers reinfection and post-exposure interventions,
the reinfection threshold $RT$ becomes important. It corresponds to critical
transmissibility values above which there is a steep nonlinear increase
in disease prevalence, corresponding to the increase contribute of
reinfection cases to the disease load. The $RT$ for the system \eqref{modelGab_controls},
in the absence of post-exposure interventions, has been computed in \cite{Gomes_etall_2007}.


\section{Optimal control problem}
\label{sec:oc:problem}

TB control is still a common problem around the world. In order to have the desire impact,
TB control measures must be timely applied. However, economical, social and environmental
constraints are imposed to TB control measures. The ideal situation would be a minimization
of active infected individuals with the lowest cost possible.
Optimal control theory is a powerful mathematical tool that can be used
to make decisions in this situation \cite{kar_Jana_2012}.

We consider the state system \eqref{modelGab_controls}
of ordinary differential equations in $\mathbb{R}^5$
with the set of admissible control functions given by
\begin{equation*}
\Omega = \left\{ (u_1(\cdot), u_2(\cdot)) \in (L^{\infty}(0, t_f))^2 \,
| \,  0 \leq u_1 (t), u_2(t) \leq 1 ,  \, \forall \, t \in [0, t_f] \, \right\} .
\end{equation*}
Our aim is to minimize the number of active infected individuals $I$ as well as the costs
required to control the disease by treating early and persistent latent individuals,
$L_1$ and $L_2$. The objective functional is given by
\begin{equation}
\label{costfunction}
\mathcal{J}(u_1(\cdot), u_2(\cdot)) = \int_0^{t_f} \left[ W_0 I(t)
+ \frac{W_1}{2}u_1^2(t) + \frac{W_2}{2}u_2^2(t) \right] dt \, ,
\end{equation}
where the constants $W_i$, $i=1, 2$, are a measure
of the relative cost of the interventions
associated to the controls $u_1$ and $u_2$, respectively,
and the constant $W_0$ is the weight constant for classe $I$.

We consider the optimal control problem of determining
$\left(S^*(\cdot), L_1^*(\cdot), I^*(\cdot), L_2^*(\cdot), R^*(\cdot)\right)$,
associated to an admissible control pair
$\left(u_1^*(\cdot), u_2^*(\cdot) \right) \in \Omega$ on the time interval $[0, t_f]$,
satisfying \eqref{modelGab_controls}, given initial conditions
$S(0)$, $L_1(0)$, $I(0)$, $L_2(0)$ and $R(0)$ and
minimizing the cost function \eqref{costfunction}, \textrm{i.e.},
\begin{equation}
\label{mincostfunct}
\mathcal{J}(u_1^*(\cdot), u_2^*(\cdot))
= \min_{\Omega} \mathcal{J}(u_1(\cdot), u_2(\cdot)) \, .
\end{equation}

In \ref{app:theo:A} we prove the following existence and uniqueness result.

\begin{theorem}
\label{the:thm}
Problem \eqref{modelGab_controls}--\eqref{mincostfunct} with given initial conditions
$S(0)$, $L_1(0)$, $I(0)$, $L_2(0)$ and $R(0)$ and fixed final time $t_f$, admits an unique
optimal solution $\left(S^*(\cdot), L_1^*(\cdot), I^*(\cdot), L_2^*(\cdot), R^*(\cdot)\right)$
associated to an optimal control pair $\left(u_1^*(\cdot), u_2^*(\cdot)\right)$ on $[0, t_f]$.
\end{theorem}

The optimal control pair predicted by Theorem~\ref{the:thm}
represents the optimal intervention strategy, given the cost constraints,
and can be found by application of the celebrated Pontryagin maximum
principle \cite{Pontryagin_et_all_1962} (Lemma in \ref{app:theo:A})
and appropriate numerical methods \cite{MyID:287}.


\section{Numerical results and cost-effectiveness analysis}
\label{sec:numericalresults}

Different approaches were used to obtain and confirm the numerical results.
One approach consisted in using IPOPT \cite{IPOPT}
and the algebraic modeling language AMPL \cite{AMPL}.
A second approach was to use the PROPT Matlab Optimal Control Software \cite{PROPT}.
The results coincide with the ones obtained by an iterative method that consists
in solving the system of ten ODEs given by \eqref{modelGab_controls} and \eqref{adjoint_function}
(Lemma in \ref{app:theo:A}). For that, first we solve system \eqref{modelGab_controls} with a guess for the controls
over the time interval $[0, T]$ using a forward fourth-order Runge--Kutta scheme
and the transversality conditions $\lambda_i(T) = 0$, $i=1, \ldots, 5$.
Then, system \eqref{adjoint_function} is solved by a backward fourth-order Runge--Kutta scheme
using the current iteration solution of \eqref{modelGab_controls}. The controls are updated
by using a convex combination of the previous controls and the values from \eqref{optcontrols}.
The iteration is stopped when the values of the unknowns at the previous iteration
are very close to the ones at the present iteration.

In the following sections all parameters are fixed according to Table~\ref{parameters},
with exception to the transmission parameter $\beta$ and the reinfection parameter
for treated individuals $\sigma_R$, which are varied to illustrate different scenarios.
The initial conditions are obtained as the nontrivial equilibria values for the system
\eqref{modelGab_controls} with no controls ($u_1=0=u_2$), corresponding to the population
state before the introduction of post-exposure interventions.


\subsection{An example of optimal control for a period of five years}
\label{subsec:example}

For illustration, we fix all parameters according to Table~\ref{parameters}.
We start by considering  $\beta=100$ and the simplest case where latent
($L_1$ and $L_2$) and recovered ($R$) individuals have the same protection
against reinfection, i.e., $\sigma_R=\sigma$. Both these assumptions will
be relaxed latter on, in Sections~\ref{subsec:beta} and \ref{subsec:sigma}.
Initial conditions are given in Table~\ref{icbeta100}.
\begin{table}[!htb]
\centering
\begin{tabular}{|r |  r | r | r | r | r |}
\hline
 {\normalsize{$S(0)$}} & {\normalsize{$L_1(0)$}} & {\normalsize{$I(0)$}} & {\normalsize{$L_2(0)$}} & {\normalsize{$R(0)$}} \\
\hline
 {\normalsize{$4\ 554$}} & {\normalsize{$ 72$}} & {\normalsize{$ 24$}} & {\normalsize{$ 23\ 950$}} & {\normalsize{$ 1\ 400$}} \\
\hline
\end{tabular}
\caption{Initial conditions for system \eqref{modelGab_controls} with parameters according
to Table~\ref{parameters} and for $\beta=100$ and $\sigma_R=\sigma$. The values are obtained
as the endemic equilibria values for \eqref{modelGab_controls} before the introduction
of post-exposure interventions (i.e., $u_1=0=u_2$).}
\label{icbeta100}
\end{table}
The solution for the optimal control problem is illustrated in
Figure~\ref{fig:cont:I:Eff:beta100:sigual} (a) and (b). During
the five years, for which the interventions lasts, the number
of infectious individuals decreases and both interventions can
be relaxed along time. Treatment intensity of the persistent latent individuals
$u_2$ must be maximum during the initial 2 years and then can be progressively reduced.
Treatment of early latent individuals $u_1$ should stay longer
at its maximum intensity, for approximately 4 years.
Figure~\ref{fig:cont:I:Eff:beta100:sigual} (c)
shows the efficacy function defined by
\begin{equation}
\label{eff}
E(t) =\frac{I(0)-I^*(t)}{I(0)}= 1 - \frac{I^*(t)}{I(0)},
\end{equation}
where $I^*(t)$ is the optimal solution associated to the optimal controls
and $I(0)$ is the corresponding initial condition. This function measures
the proportional decrease in the number of infectious individuals imposed
by the intervention with controls $(u_1, u_2)$, by comparing the number
of infected individuals at time $t$ with the initial value $I(0)$ for which
there are no controls implemented ($u_1 = u_2 =0$). By construction,
$E(t) \in [0, 1]$ for all time $t$ and the efficacy is highest when $E(t)$ is one.
Note that $E(t)$ has the contrary tendency of $I(t)$.
\begin{figure}[ht]
\centering
\subfloat[]
{\label{controls:beta100:sigual}
\includegraphics[scale=0.37]{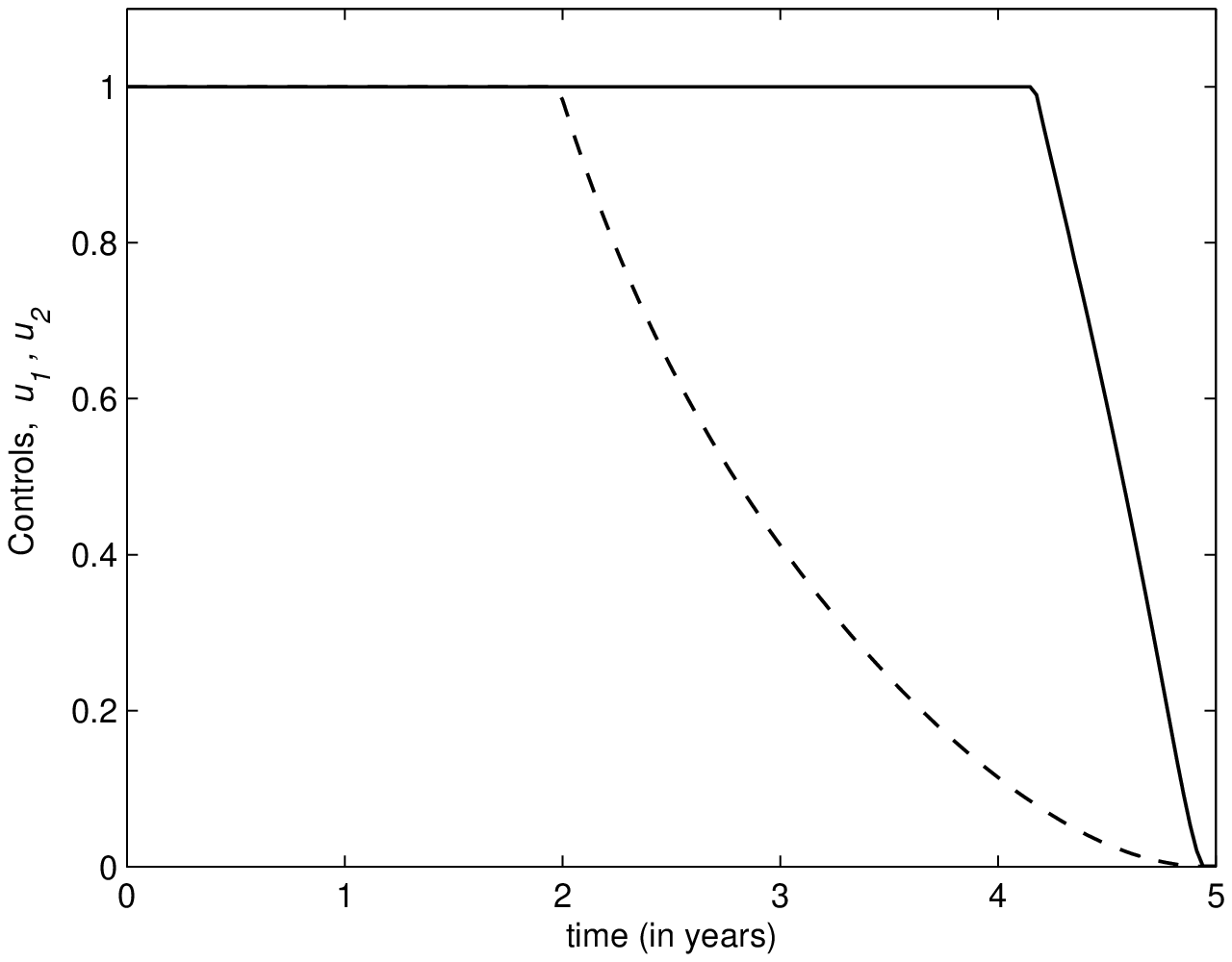}}
\subfloat[]
{\label{Inf:beta100:sigual}
\includegraphics[scale=0.37]{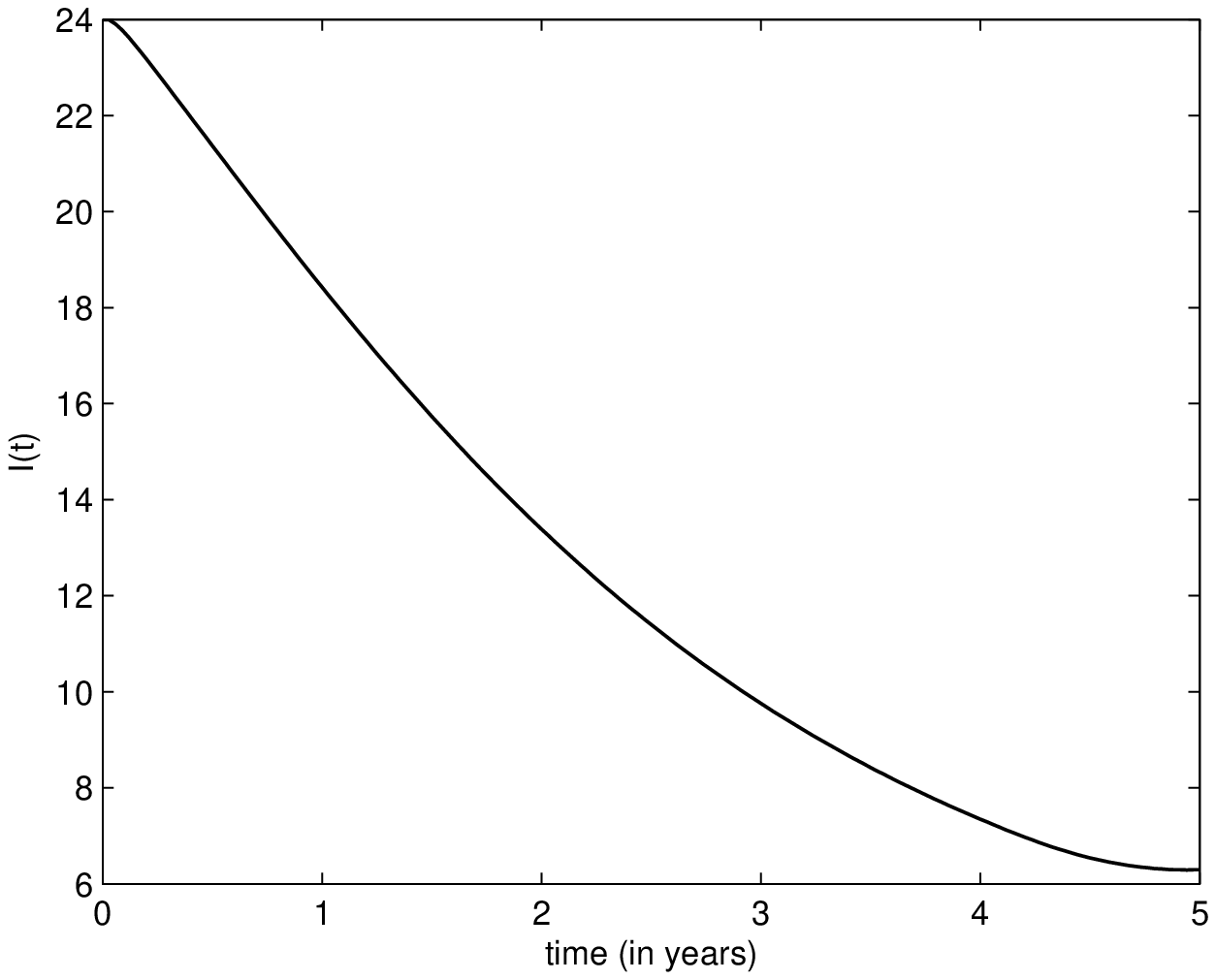}}
\subfloat[]
{\label{Eff:beta100:sigual}
\includegraphics[scale=0.37]{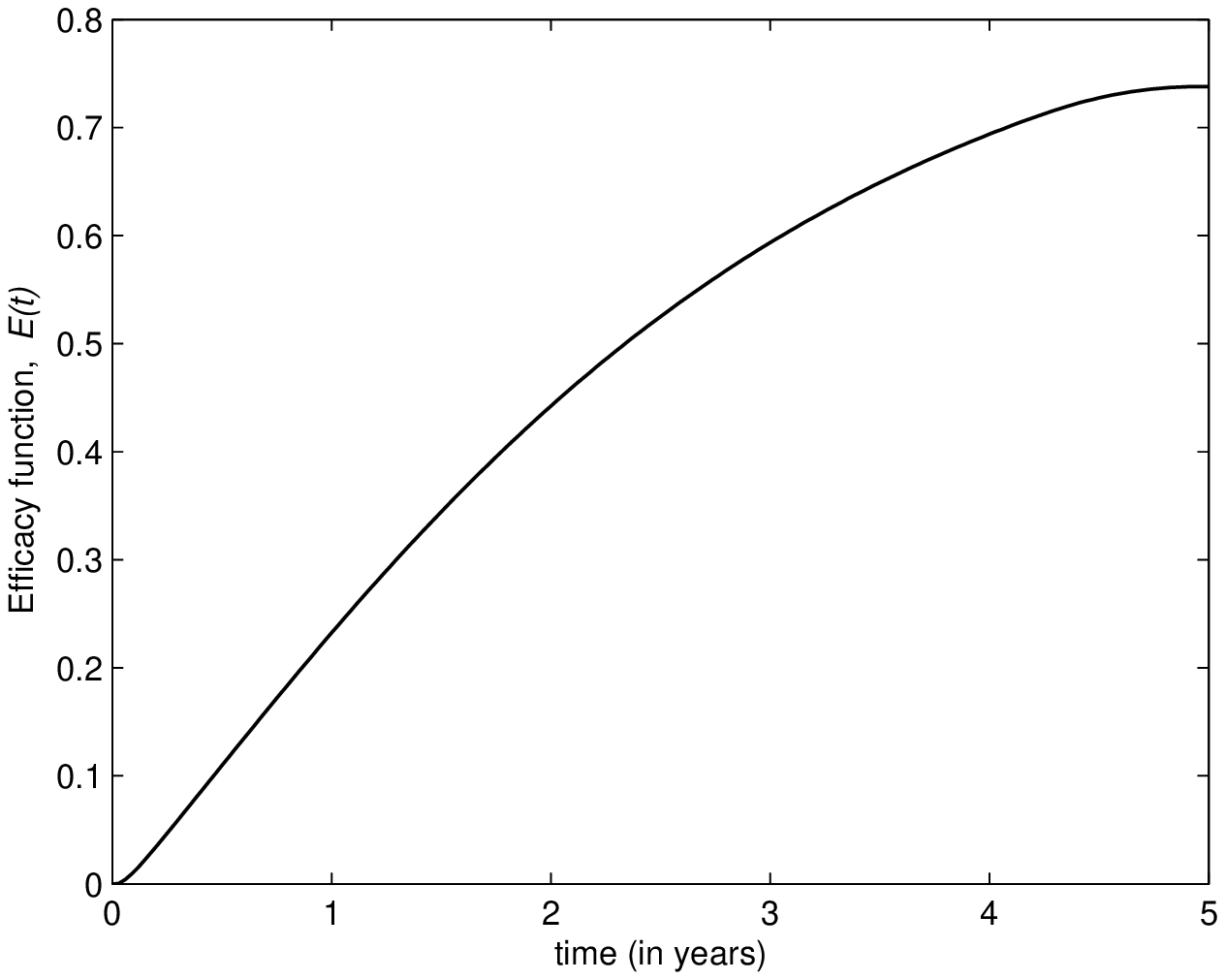}}
\caption{Solution for the optimal control problem \eqref{modelGab_controls}--\eqref{mincostfunct},
assuming $t_f=5$, $\beta = 100$ and $\sigma_R = \sigma$. {\bf (a)} Optimal control
pair $u_1$ (continuous line) and $u_2$ (dashed line).
{\bf (b)} Number of infectious individuals along time.
{\bf (c)} Efficacy function $E(t)$, defined by \eqref{eff}.}
\label{fig:cont:I:Eff:beta100:sigual}
\end{figure}

Naturally, the results depend on the objective function $\mathcal{J}$ given by \eqref{costfunction}.
In particular, they depend on the duration of the intervention $t_f$ and on the weight
constants associated with the amount of infectious individuals $W_0$
and with the costs of controls $W_i$, $i=1,2$. WHO goals are usually fixed for five years periods,
so in what follows, we assume $t_f=5$ years. Moreover, for higher values of $t_f$
($t_f \in \{10,\ldots, 25 \}$) we can observe that the number of infected individuals
starts to increase towards the end of the intervention (\ref{app:tf:B}). In practical terms,
this would mean that the intervention should be revised before its end. Results do not change qualitatively
by varying constants $W_i$, $i=0,1,2$. However, the magnitude of the efficacy changes more significantly
in the cases where $W_0$ and $W_1=W_2$ are varied independently. Generally, efficacy decreases when the costs
$W_1$ and $W_2$ increase, corresponding to earlier relaxation of the intensity of treatment ($u_1(t), u_2(t)$)
in the optimal solution. More details can be found in \ref{app:Wis:C}.

More importantly, these results will change depending on the epidemiological scenario we consider.
In the next subsections we vary the transmission coefficient $\beta$
and on the protection conferred by treatment $\sigma_R$.


\subsection{Summary measures}
\label{subsec:smeasures}

We introduce some summary measures to evaluate the cost and the effectiveness
of the proposed control measures for the entire intervention period,
for different epidemiological scenarios.

For each $\beta$ and $\sigma_R$ fixed, the total cases averted by the intervention
during the time period $t_f$ is given by
\begin{equation}
\label{A:beta:sigmaR}
A(\beta, \sigma_R)=t_f I(0; \beta, \sigma_R)-\int_0^{t_f} I^*(t; \beta, \sigma_R) dt,
\end{equation}
where, for each $\beta$ and $\sigma_R$ fixed, $I^*(t; \beta, \sigma_R)=I^*(t)$
is the optimal solution associated to the optimal controls ($u_1^*, u_2^*$) and
$I(0;\beta,\sigma_R)=I(0)$ is the corresponding initial condition.
Note that this initial condition is obtained as the equilibrium proportion $\overline{I}(\beta,\sigma_R)$
of system \eqref{modelGab_controls} with no post-exposure intervention ($u_1=u_2=0$),
which does not depend on time, so $t_f I(0; \beta, \sigma_R)=\int_0^{t_f}\overline{I}(\beta,\sigma_R)dt$
represents the total infectious cases over a period of $t_f$ years.

We define effectiveness as the proportion of cases averted
on the total cases possible under no intervention:
\begin{equation}
\label{eq:Ebarra}
\overline{E}(\beta, \sigma_R) =\frac{A(\beta, \sigma_R)}{t_f I(0; \beta, \sigma_R)}
= 1 - \frac{\displaystyle\int_0^{t_f} I^*(t; \beta, \sigma_R) dt}{t_f I(0; \beta, \sigma_R)}.
\end{equation}
We choose dimensionless measures for effectiveness
to be able to compare different epidemiological scenarios.

The total cost associated to the intervention is
\begin{equation}
\label{eq:TC}
TC(\beta,\sigma_R) = \int_0^{t_f} C_1 u^*_1(t)L^*_1(t) + C_2u_2^*(t)L^*_2(t) dt,
\end{equation}
where $C_i$ correspond to the per person unit cost of the two possible interventions:
detection and treatment of early latent individuals ($C_1$)
and chemotherapy/vaccination of persistent latent individuals ($C_2$).
Following \cite{Okosun_etall_2013}, we define the average cost-effectiveness ratio by
\begin{equation}
\label{A}
ACER=\frac{TC}{A}.
\end{equation}
Typically, optimal solutions correspond to maximum intensity of intervention
for a certain period followed by relaxation, as in the example in Section~\ref{subsec:example}.
So, we use the time at which the intensity of each intervention is relaxed as another way
to evaluate the effort associated with an optimal solution:
\begin{equation*}
tr_{i}=tr_{i}(\beta, \sigma_R)=\max \{t\in[0,t_f]: u_i(t;\beta,\sigma_R)=1 \},
\quad i=1,2.
\end{equation*}
We refer to these as {\it relaxation-times}.
Table~\ref{table:sm:beta100:sigual} summarizes the particular case
analyzed in the previous section, $\beta=100$ and $\sigma_R=\sigma$.
\begin{table}[!htb]
\centering
\begin{tabular}{|c|c|c|c|c|c|c|}
\hline
$\beta$ & {\normalsize{$\displaystyle\overline{A}$}} & {\normalsize{$TC$}}& {\normalsize{$ACER$}}
&  {\normalsize{$\overline{E}$}} & {\normalsize{$t_{r_1}$}} &  {\normalsize{$t_{r_2}$}} \\ \hline
{\normalsize{100}}   & {\normalsize{$56$}} & {\normalsize{$ 23\ 374$}} & {\normalsize{$417$}}
& {\normalsize{$0.4691 $}} & {\normalsize{$4.1765$}} & {\normalsize{$2.0$}}\\ \hline
\end{tabular}
\caption{Summary of cost-effectiveness measures for $\beta=100$ and $\sigma_R = \sigma$.}
\label{table:sm:beta100:sigual}
\end{table}


\subsection{Impact of transmission intensity on optimal control interventions}
\label{subsec:beta}

First we compare model results for different epidemiological scenarios
in terms of transmission intensity, by varying parameter $\beta$. For now,
we assume that protection conferred by natural infection or by treatment
is the same ($\sigma_R=\sigma$). The remaining parameters
are fixed according to Table~\ref{parameters}.

Figure~\ref{fig:Ebar:Cbar:Tri:betavariar:sigual} represents effectiveness $\overline{E}$
and relaxation-times $t_{r_i}$, $i=1,2$, for the optimal control measures,
when varying transmission intensity $\beta$.
\begin{figure}[!ht]
\centering
\subfloat[\footnotesize{$\overline{E}$ for variable  $\beta$ ($\sigma_R = \sigma$).}]
{\label{Ebar:betavariar:sigual}
\includegraphics[scale=0.65]{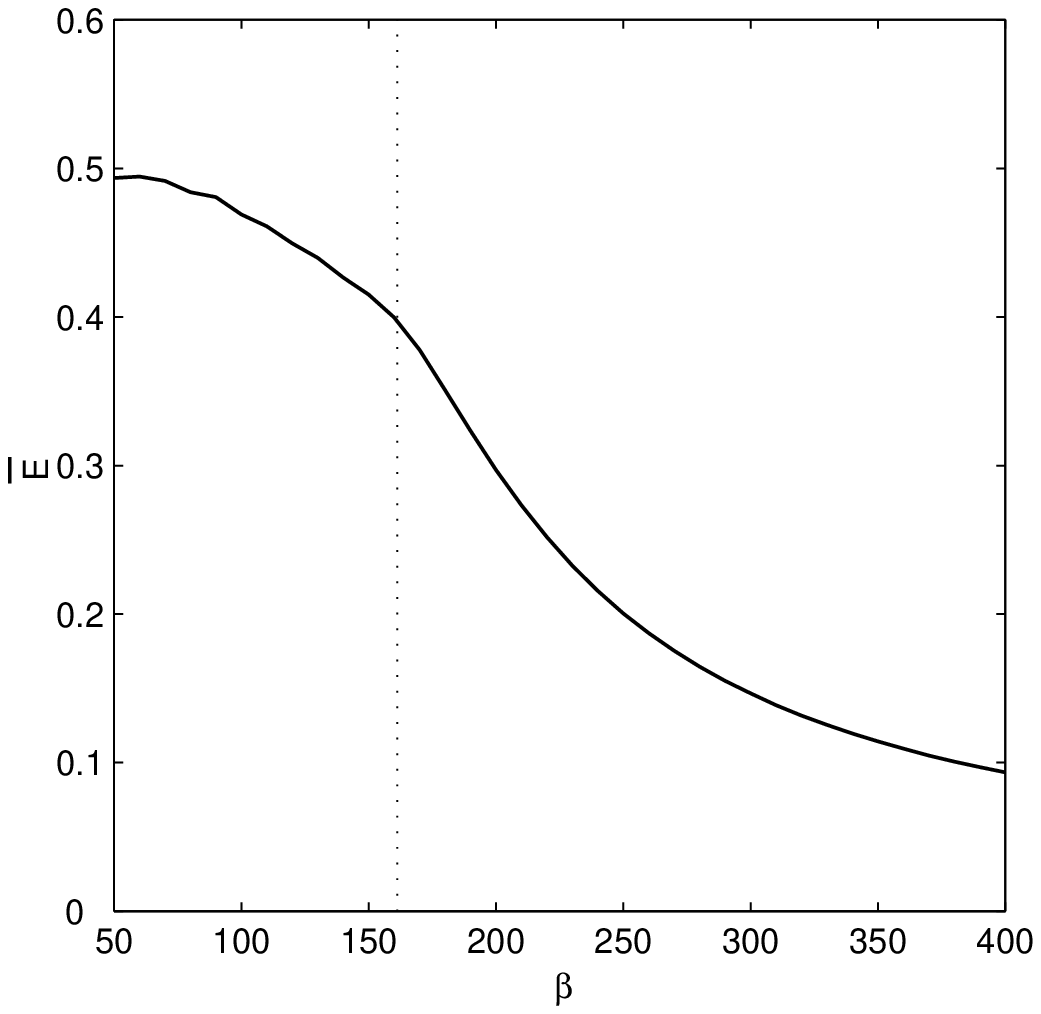}}
\subfloat[\footnotesize{$t_{r_i}$ for variable  $\beta$ ($\sigma_R = \sigma$).}]
{\label{Tri:betavariar:sigual}
\includegraphics[scale=0.65]{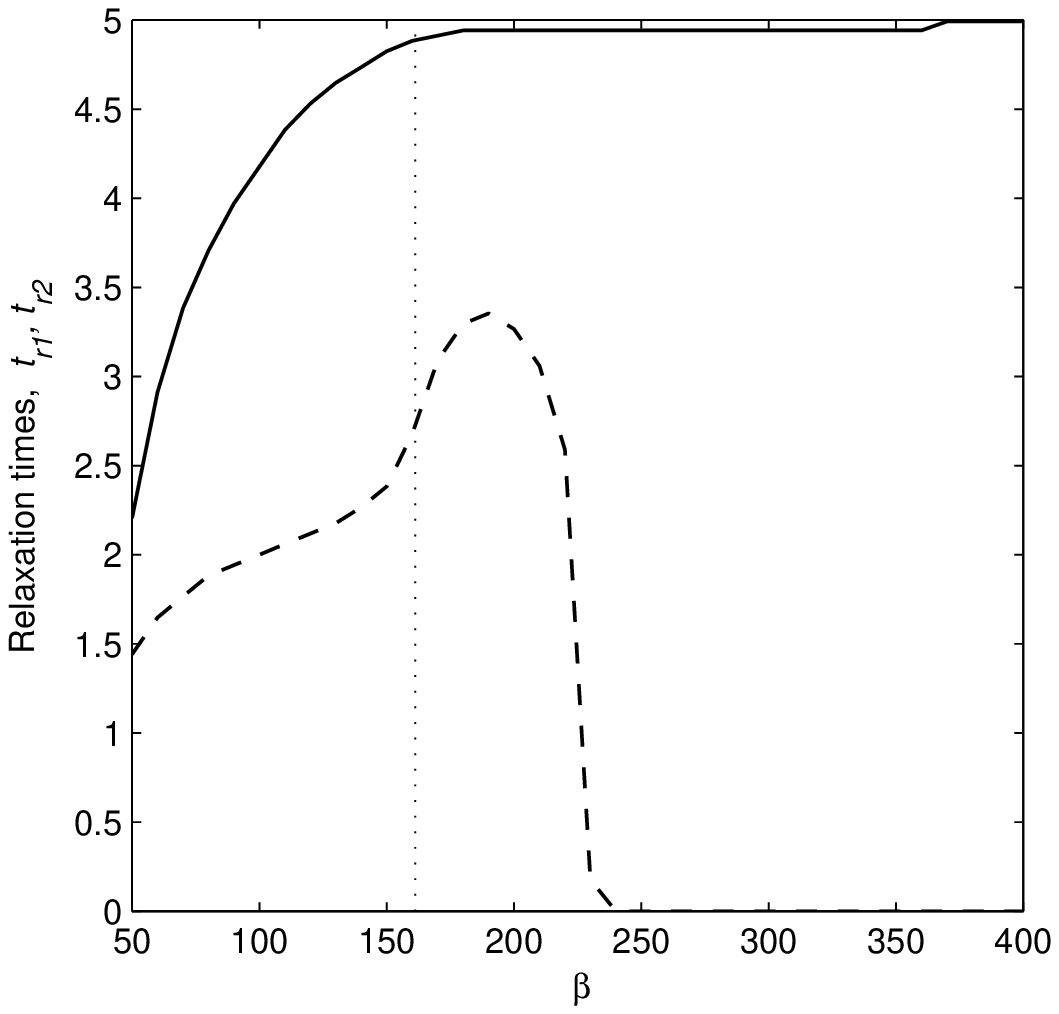}}
\caption{{\bf (a)} Effectiveness $\overline{E}$, and {\bf (b)} Relaxation-times
($t_{r_1}$ full line and $t_{r_2}$ dashed line), for variable $\beta$ and $\sigma_R = \sigma$.
Parameters according to Table~\ref{parameters}.}
\label{fig:Ebar:Cbar:Tri:betavariar:sigual}
\end{figure}
Effectiveness is a monotonically decreasing function on $\beta$. The reinfection threshold $RT$,
marked by the dotted vertical line, coincides with a change in curvature of $\overline{E}(\beta)$
from concave to convex (Figure~\ref{Ebar:betavariar:sigual}). For all endemic scenarios,
maximum intensity of treatment of early latent individuals is required for longer periods
than treatment of persistent latent individuals (Figure~\ref{Tri:betavariar:sigual}). Below the $RT$,
the relaxation-times of both post-exposure interventions increase with $\beta$. However,
above the $RT$, treatment of early latent individuals is required at its maximum intensity
for almost the entire five year period ($t_f$) and the intervention on persistent latent individuals
is needed for shorter periods. For very high transmission intensity, relaxation time for intervention
on persistent latent individuals is zero ($t_{r_2}=0$), corresponding to a singular control.

Depending on the background epidemiological scenario, we can have different optimal intervention strategies.
For example, for $\beta=100$ the optimal solution corresponds to both interventions with relaxation-times
of $t_{r_1}= 4.1765$ and $t_{r_2}=2.0$ years and for $\beta=250$ the optimal solution corresponds to treatment
of early latent individuals for approximately the entire intervention period, $t_{r_1}=4.941$ years and
treat persistent individuals at intensity always below the maximum $u_2^*(t)<1$, for $t \in [0, t_f]$
(results not shown). These interventions are associated with very different effectiveness,
45\% ($\overline{E}(100)=0.4691$) and 20\% ($\overline{E}(250)=0.2005$), respectively.


\subsection{Impact of protection against reinfection of the treated individuals
($\sigma_R\neq\sigma$) on optimal control interventions}
\label{subsec:sigma}

In this section we relax the assumption that latent ($L_1$ and $L_2$) and treated ($R$) individuals
have the same protection to reinfection. Given the lack of published studies supporting on of the hypothesis,
we explore both possibilities, as in \cite{Gomes_etall_2007}: treatment enhances protection against reinfection
($\sigma_R<\sigma$) or protection is impaired by treatment ($\sigma_R>\sigma$).
To illustrate, we will use $\sigma_R=\sigma/2$ and $\sigma_R=2\sigma$, respectively.
\begin{figure}[!ht]
\centering
\subfloat[]
{\label{Ebar:betavariar:sdif}
\includegraphics[scale=0.50]{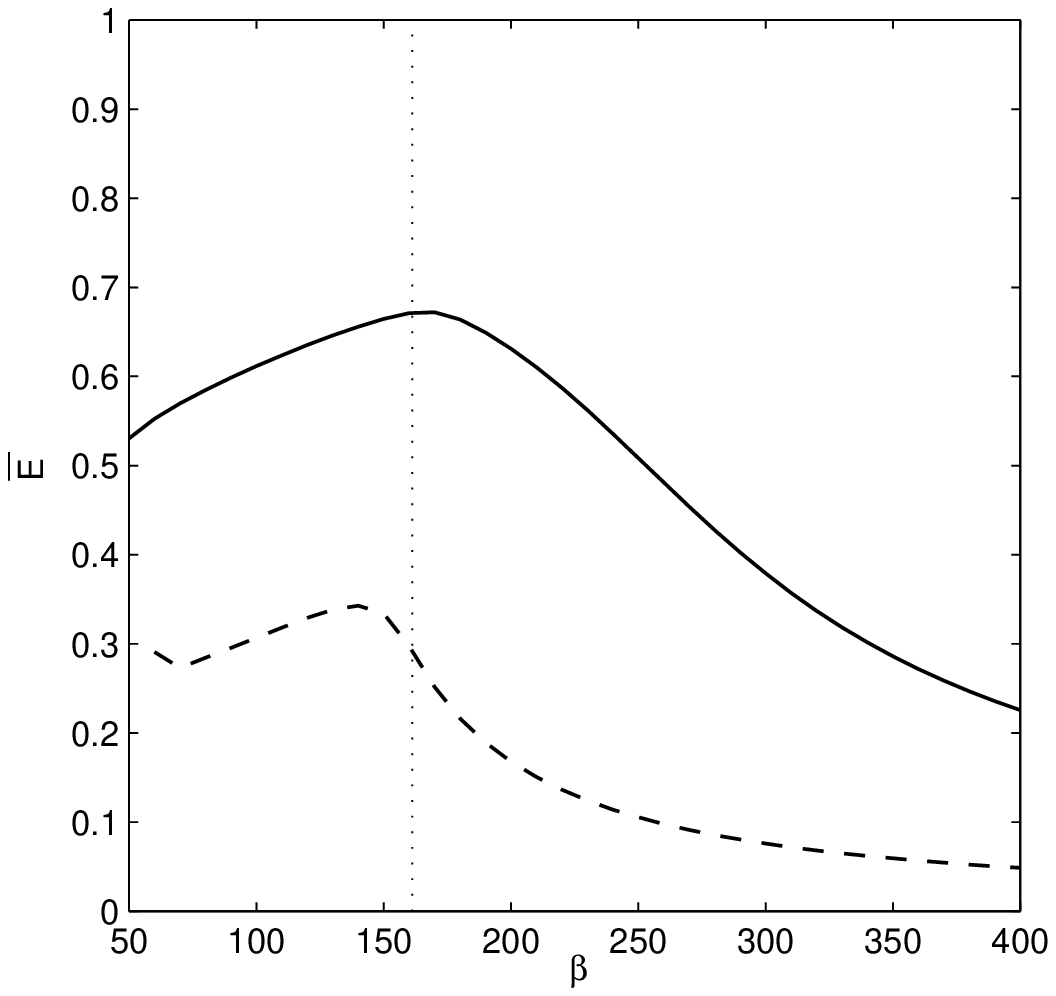}}
\subfloat[]
{\label{Tr1:betavariar:sdif}
\includegraphics[scale=0.45]{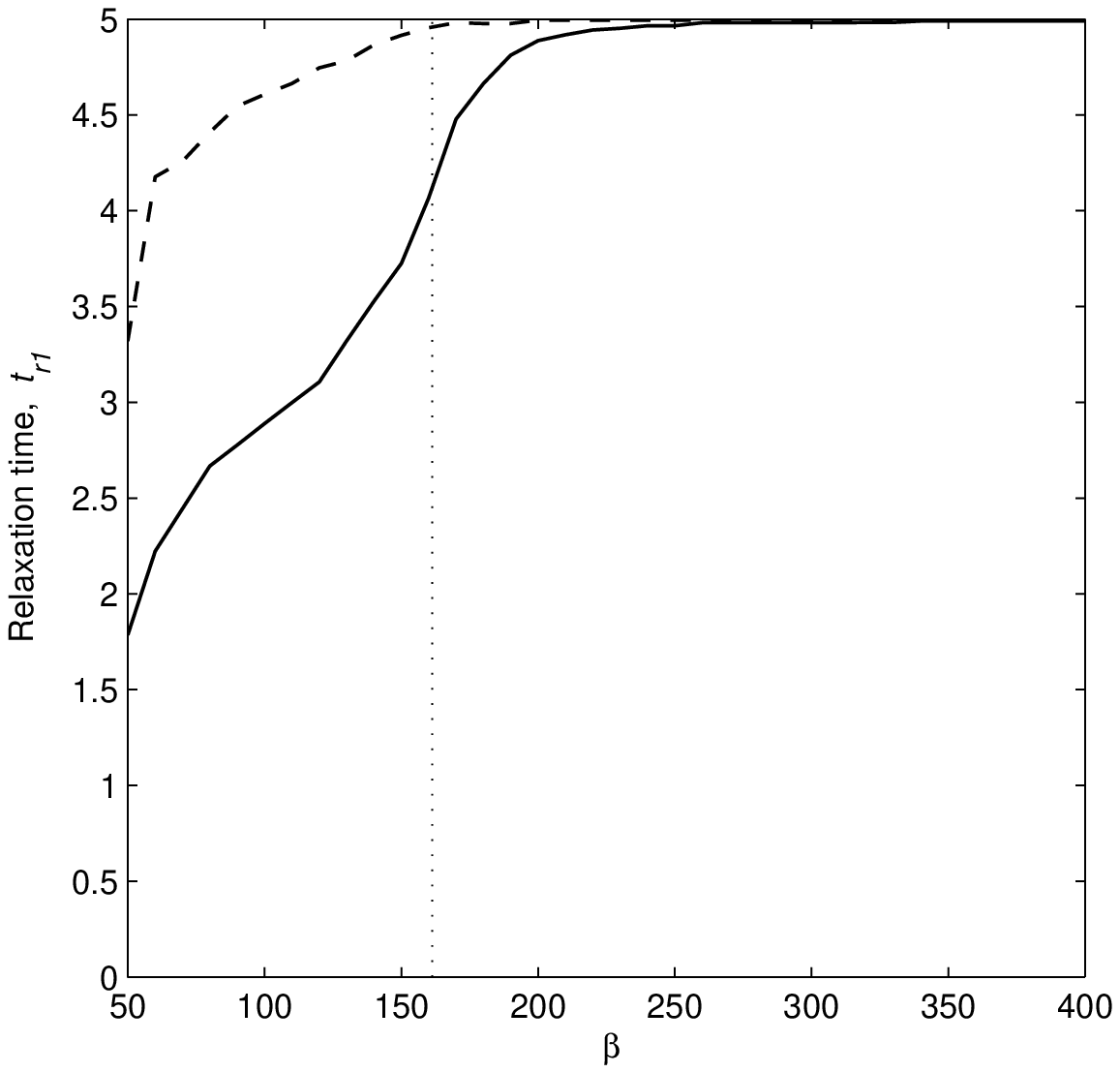}}
\subfloat[]
{\label{Tr2:betavariar:sdif}
\includegraphics[scale=0.45]{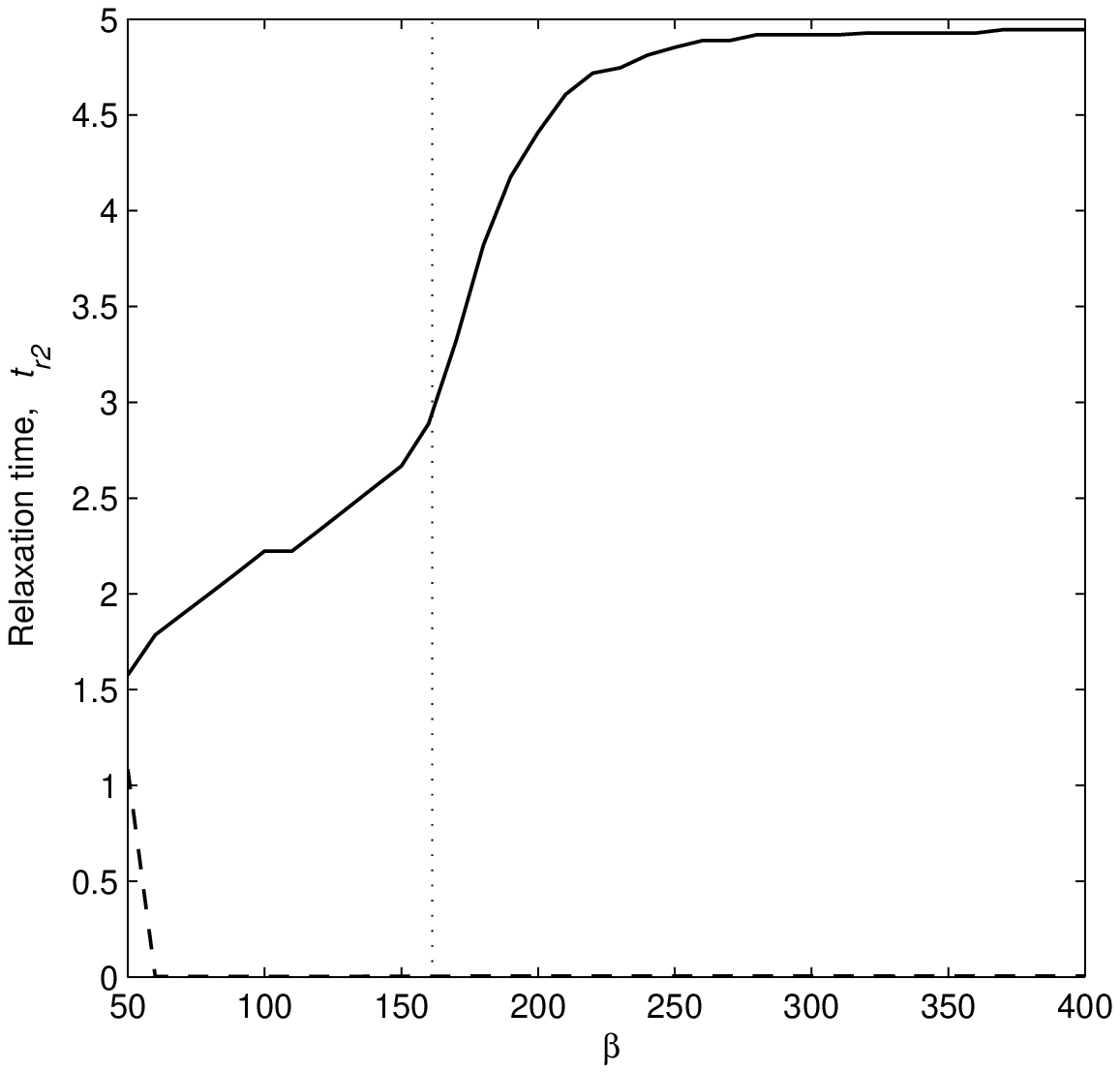}}
\caption{{\bf (a)} Effectiveness $\overline{E}$, {\bf (b)}--{\bf (c)}
Relaxation-times $t_{r_1}$ and $t_{r_2}$, for variable $\beta$.
Full and dashed lines correspond to cases $\sigma_R = \sigma/2$ and $\sigma_R=2\sigma$, respectively.
Parameters according to Table~\ref{parameters}.}
\label{fig:Ebar:Tr1:Tr2:betavariar:sdif}
\end{figure}
Results are very different for the two scenarios. If protection against reinfection
is enhanced by treatment ($\sigma_R=\sigma/2$), then the optimal solution corresponds
to treat both early and persistent latent individuals at maximum intensity
for a certain period, ranging from 1.5 to 5 years, followed by relaxation
of the intervention intensity (full lines in Figure~\ref{Tr1:betavariar:sdif}
and \ref{Tr2:betavariar:sdif}).
\begin{figure}[!ht]
\centering
\subfloat[]
{\label{fig:u1}\includegraphics[scale=0.58]{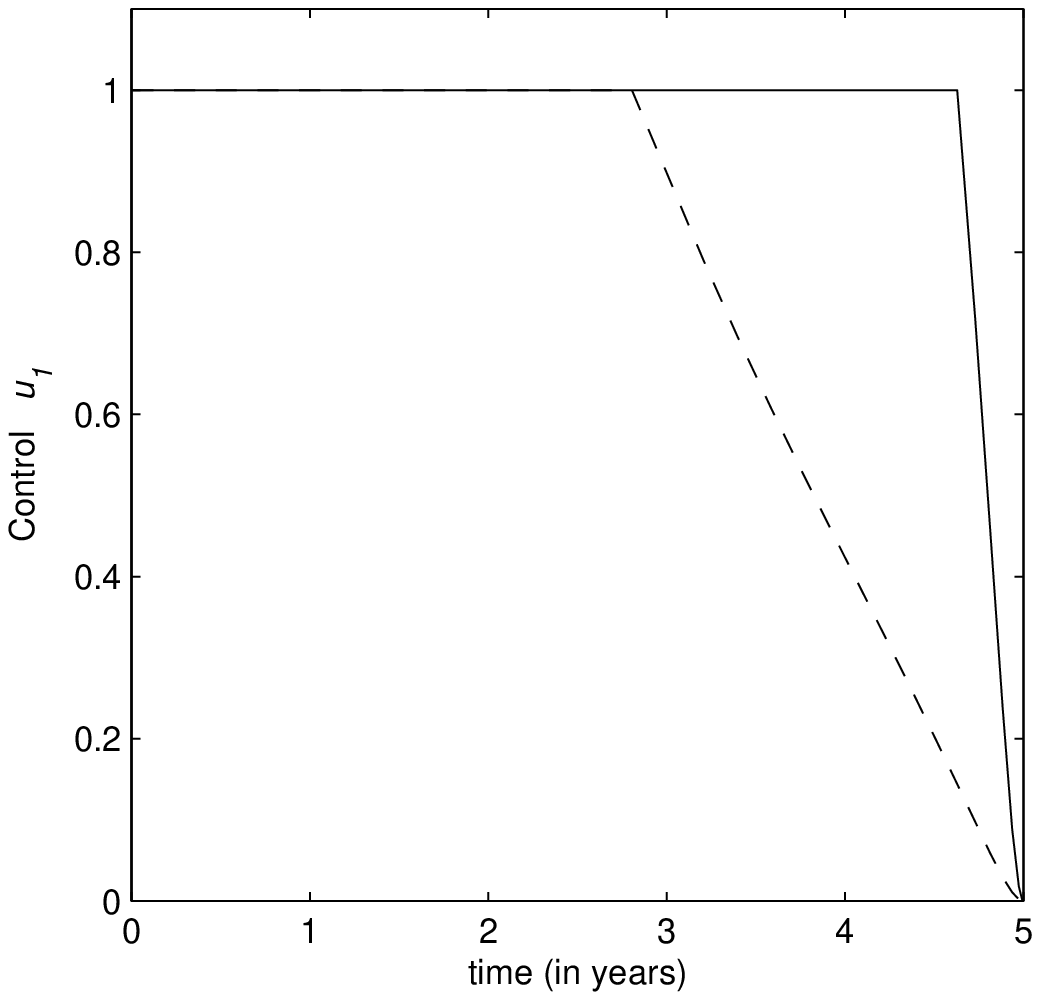}}
\subfloat[]
{\label{fig:u2}\includegraphics[scale=0.58]{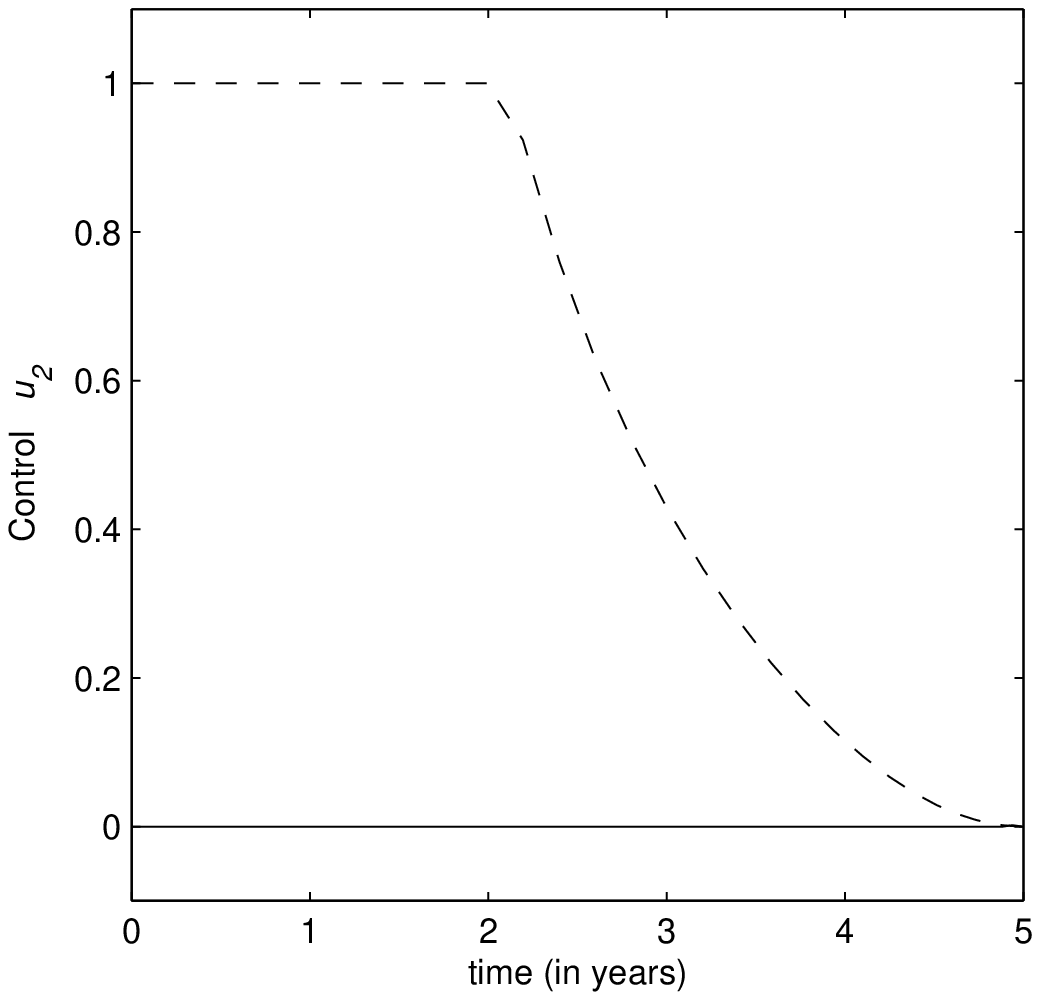}}
\caption{{\bf (a)} Control $u_1$, {\bf (b)} Control $u_2$ for $\beta=100$.
Full and dashed lines correspond to cases $\sigma_R = \sigma/2$ and $\sigma_R=2\sigma$,
respectively. Parameters according to Table~\ref{parameters}.}
\label{fig:beta100:allu:controls}
\end{figure}
The relaxation-times increase with $\beta$. However, if treatment impairs protection,
then the optimal intervention would be to treat early latent at maximum intensity
for longer periods and to treat persistent latent individuals almost always
below the maximum intensity (dashed lines in Figure~\ref{Tr1:betavariar:sdif}
and \ref{Tr2:betavariar:sdif}).
Actually, in this case the optimal solution can impose not to treat persistent latent individuals
($u_2^*=0$ for $t \in [0, t_f]$) as illustrated in Figure~\ref{fig:beta100:allu:controls}
for the case $\beta=100$. In both cases, effectiveness peaks close to the reinfection threshold $RT$.


\subsection{Optimal controls strategy and cost-effectiveness analysis}
\label{subsec:oc:strategies}

In this section we analyse the cost-effectiveness of alternative combinations
of the two possible control measures: strategy {\bf a} --  implementing both controls
$u_1$ and $u_2$, corresponding to intervene on both early and persistent latent individuals,
as in previous sections; strategy {\bf b} -- implementing only control measure $u_1$;
and strategy {\bf c} -- only control measure $u_2$, separately.

For each value of $\beta$, we compute the optimal solution for the three strategies and calculate
the associated effectiveness $\overline{E}$. In Figure~\ref{Ebar:betavariar:sigual_ui} we can see that,
below the reinfection threshold $RT$, the strategy using interventions on both population groups has higher effectiveness.
However, above the $RT$ this advantage is marginal, comparing with the intervention on early latent individuals, only.
\begin{figure}[!ht]
\centering
\subfloat[]
{\label{Ebar:betavariar:sigual_ui}
\includegraphics[scale=0.58]{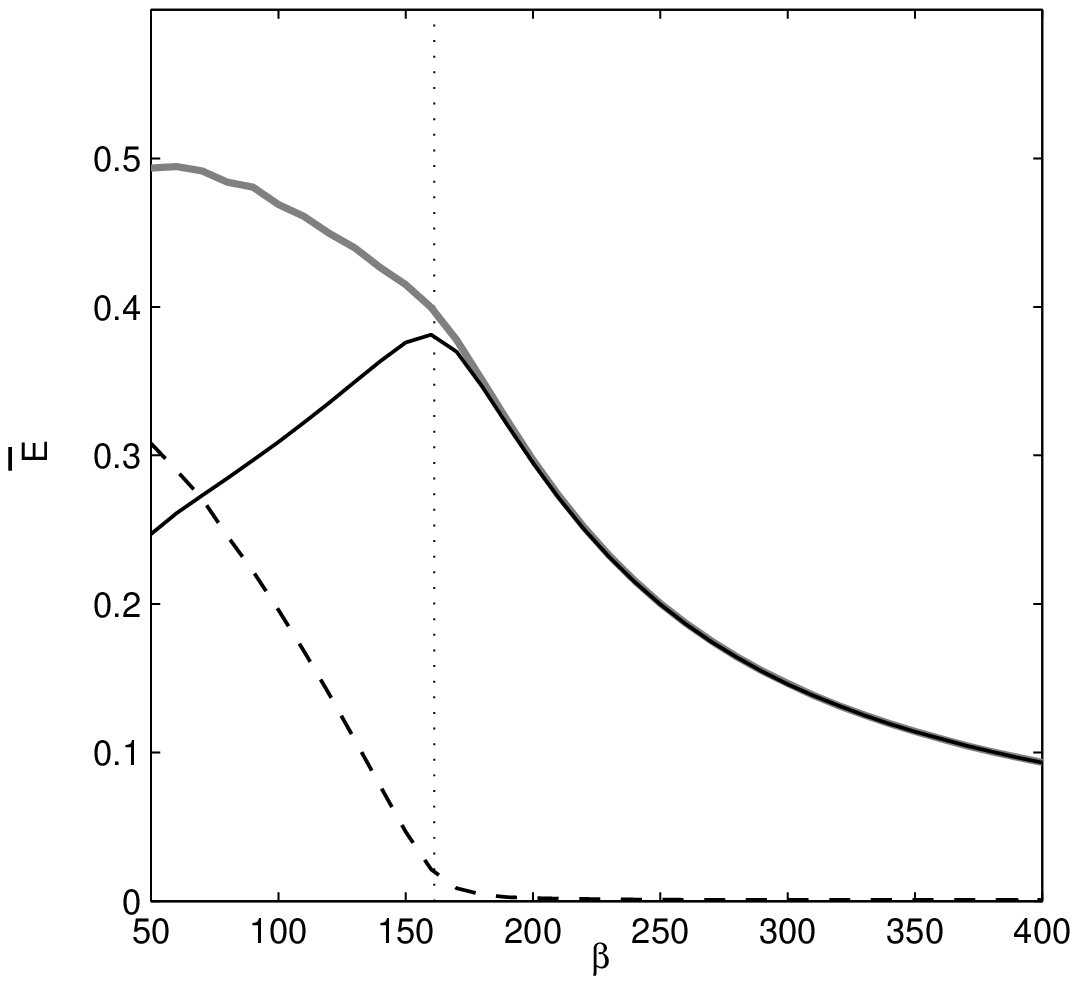}}
\subfloat[]{\label{Tr:betavariar:sigual_ui}
\includegraphics[scale=0.58]{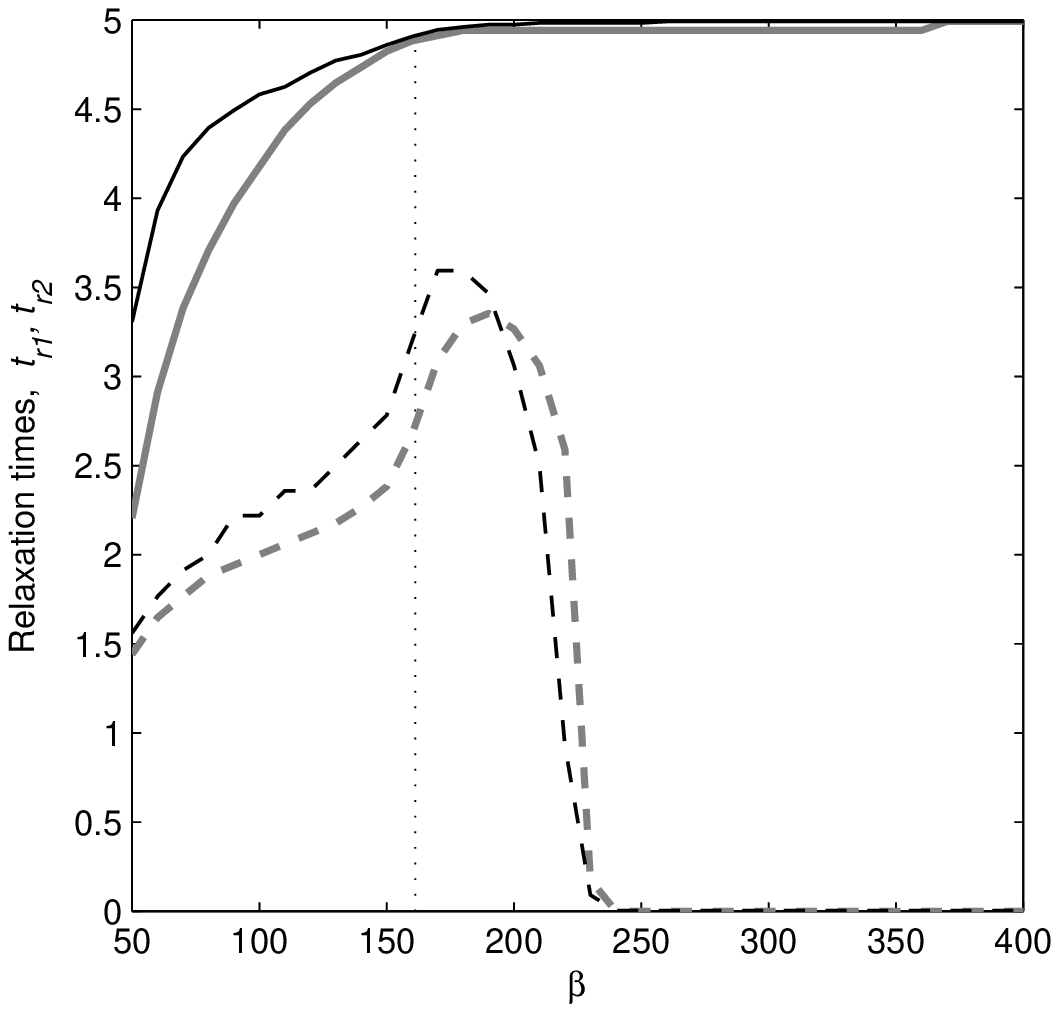}}
\caption{{\bf (a)} Effectiveness $\overline{E}$ and {\bf (b)} Relaxation-times,
$t_{r_1}$ or $t_{r_2}$, for variable $\beta$. Grey lines correspond to intervention
with both controls $u_1$ and $u_2$. Black full and dashed lines correspond to interventions
with only $u_1$ or $u_2$, respectively. Parameters according to Table~\ref{parameters} and $\sigma_R = \sigma$.}
\label{EbarTr:betavariar:sigual_ui}
\end{figure}
From Figure~\ref{Tr:betavariar:sigual_ui}, we can have one idea of the time design
of the optimal intervention for each case. Intervention on early latent individuals only,
corresponds to a control at maximum intensity for long periods (very high $t_{r_1}$).
When intervening on persistent latent individuals only, maximum intensity of control
is required for shorter periods which are close to zero for very high transmission intensity,
corresponding to singular controls.

For a particular epidemiological scenario (fixed $\beta$ and $\sigma_R$), we can use
a more classical approach to analyse the cost-effectiveness of the three alternative
strategies by using the incremental cost effectiveness ratio (ICER) in \cite{Okosun_etall_2013}.
This ratio is used to compare the differences between the costs and health outcomes
of two alternative intervention strategies that compete for the same resources
and it is generally described as the additional cost per additional health outcome.
First, we must rank the strategies in order of increasing effectiveness, here measured
as the total infections averted $A(\beta,\sigma_R)$, defined in \eqref{A}. Given two
competing strategies $a$ and $b$, the ICER of the strategy with the least effectiveness
is its ACER and for the following strategies is given by
\begin{equation*}
ICER(b)=\frac{A(b)- A(a)}{TC(b)-TC(a)}.
\end{equation*}
For illustration, we focus on an epidemiological scenario of moderate transmission
with $\beta=100$. Results are shown in Table~\ref{table:ICER:100}. Strategy {\bf c}
has a unit cost of 1\ 485, it is more costly and less effective than strategy {\bf b},
so we exclude strategy {\bf c} from the set of alternatives.
\begin{table}[!htb]
\centering
\begin{tabular}{|c|c|c|c|c|}
\hline
{\normalsize{Strategy}} & {\normalsize{$A$}} & {\normalsize{$TC$}}& {\normalsize{$ACER$}} & {\normalsize{$ICER$}} \\
\hline
c &  {\normalsize{$24$}} & {\normalsize{$35\ 640$}}& {\normalsize{$1\  485$}}& {\normalsize{$1\ 485$}} \\
\hline
b &  {\normalsize{$37$}} & {\normalsize{$211$}}& {\normalsize{$5.7$}}& {\normalsize{$-1\ 721$}} \\
\hline
a &  {\normalsize{$56$}} & {\normalsize{$23\ 374$}}& {\normalsize{$417.4$}}& {\normalsize{$1\ 207$}} \\
\hline
\end{tabular}
\caption{Incremental cost-effectiveness ratio for alternative strategies a, b and c, with $\beta=100$.
Parameters according to Table~\ref{parameters}, $C_1=C_2=1$ and $\sigma_R = \sigma$.}
\label{table:ICER:100}
\end{table}
We align the remaining alternative strategies by increasing effectiveness and recompute the ICER:
ICER(b)=ACER(b)=5.7 and ICER(a)=$1\ 207$. Hence, we conclude that strategy {\bf b} has the least
ICER and therefore is more cost-effective than strategy {\bf a}. For this illustration
we have considered the same cost for both interventions ($C_1=C_2=1$). Results should depend
strongly on the choice of these parameters, however this discussion is out of the scope of our present work.


\section{Discussion}
\label{sec:discussion}

In this work we study the potential of widespread of two post-exposure interventions that are not widely used:
treatment of early latent individuals and prophylactic treatment/vaccination of persistent latent individuals.
We propose an optimal control problem that consists in analysing how these two control measures should be implemented,
for a certain time period, in order to reduce the number of active infected individuals, while controlling
the interventions implementation costs. This approach differs from others
\cite{Blower_Small_Hopewell_1996,Castillo_Feng_1998,Dye_et_all_1998,Gomes_etall_2007}
since it allows intensity of intervention to be changed along time.

As previous suggested \cite{Gomes_et_all_2004,Gomes_etall_2007}, interventions impact can be sensitive
to transmission intensity and reinfection. We choose a dimensionless measure of effectiveness to compare
different scenarios: assuming different transmission intensity ($\beta$) or assuming different assumptions
on protection against reinfection conferred by treatment ($\sigma_R$).

Effectiveness of optimal intervention decreases with transmission.
There is a change in the intervention profile from low to high transmission.
In high transmission settings, the intensity of treatment of persistent
latent individuals $u_2^*$ for the optimal solution is reduced.
Since treatment of persistent latent individuals reduces the reactivation rate
(from $\omega$ to $\omega_R$), when reinfection is very common and it overcomes
reactivation impact, the advantage of treating this population group is less pronounced.

The susceptibility to reinfection after treatment is still an open question.
In one hand, treatment can reduce the risk of TB by reducing the amount of bacteria present
in the lungs. On the other hand, we can argue that latent infection boosts immunity by constant
stimulation of the immune system, so treatment could reduce protection. We vary parameter
$\sigma_R$ to explore these two possible scenarios: $\sigma_R=\sigma/2$ when treatment
enhances protection and $\sigma_R=2\sigma$ when treatment impairs protection. Results
show that treatment of persistent latent individuals should be less intense or even absent
for the case where treatment impairs protection. Similar results were obtained for the case
of constant treatment rates in \cite{Gomes_etall_2007}. In fact, for the correspondent
case with maximum intensity ($u_1\equiv 1$ and $u_2\equiv 1$), we can have an increase
of the equilibrium proportion of infectious individuals ($t\rightarrow \infty$).

We can conclude that reinfection has an important role in the determination
of the optimal control strategy, by diminishing the intervention intensity
on persistent latent individuals: first when transmission is very high corresponding
to a very high reinfection rate and secondly when this population group has
a lower susceptibility to reinfection ($\sigma<\sigma_R$). Interestingly,
the reinfection threshold $RT$ of the model with no controls still marks
a change in the model behaviour. Even though, we are comparing equilibrium
results to transient short time interventions.

Cost-effectiveness analysis of alternative combinations of the two interventions is conducted.
For $\beta=100$, treatment of only early latent individuals is the more cost-effective strategy,
despite of treatment of both early latent and persistent latent individuals having a higher effectiveness.
The total cost associated with treatment of persistent latent individuals is very high,
especially because this population group can be very big in comparison to the others.
It is believed that about one third of world's population is latent infected with TB. Here,
for simplicity, we have considered the cost parameters both equal to one. However,
this depends greatly on the type of intervention used and results can be changed. For example,
if intervention on persistent latent individuals could be done by vaccination, then
the per person unit cost could be significantly reduced. Plus, treatment of early latent individuals
implies contact tracing of index cases and prophylactic treatment, which can also be very expensive.


\appendix


\section{Proof of Theorem~\ref{the:thm}}
\label{app:theo:A}

The Hamiltonian $H$ associated to the problem in \eqref{modelGab_controls} is given by
\begin{equation*}
\begin{split}
H&= H(S(t), L_1(t), I(t), L_2(t), R(t), \lambda(t), u_1(t), u_2(t)) \\
&=W_0 I(t) + \frac{W_1}{2}u_1^2(t) + \frac{W_2}{2}u_2^2(t) \\
&\quad + \lambda_1(t) \left(\mu N - \frac{\beta}{N} I(t) S(t) - \mu S(t) \right)\\
&\quad + \lambda_2(t) \left( \frac{\beta}{N} I(t)\left( S(t)
+ \sigma L_2(t) + \sigma_R R(t)\right) - (\delta + \tau_1 u_1(t) + \mu)L_1(t) \right)\\
&\quad + \lambda_3(t) \left(\phi \delta L_1(t) + \omega L_2(t)
+ \omega_R R(t) - (\tau_0 + \mu) I(t) \right)\\
&\quad + \lambda_4(t) \left((1 - \phi) \delta L_1(t)
- \sigma \frac{\beta}{N} I(t) L_2(t)
- (\omega + \tau_2 u_2(t) + \mu)L_2(t) \right)\\
&\quad  + \lambda_5(t) \left(\tau_0 I(t) + \tau_1 u_1(t) L_1(t) + \tau_2 u_2(t) L_2(t)
- \sigma_R \frac{\beta}{N} I(t) R(t) - (\omega_R + \mu)R(t) \right),
\end{split}
\end{equation*}
where $\lambda(t) = \left(\lambda_1(t), \lambda_2(t),
\lambda_3(t), \lambda_4(t), \lambda_5(t)\right)$ is
the \emph{adjoint vector}. According to the Pontryagin maximum principle
\cite{Pontryagin_et_all_1962}, if $(u_1^*(\cdot), u_2^*(\cdot)) \in \Omega$
is optimal for problem \eqref{modelGab_controls}--\eqref{mincostfunct}
with the initial conditions given in Table~\ref{icbeta100}
and fixed final time $t_f$, then there exists a nontrivial absolutely continuous mapping
$\lambda : [0, t_f] \to \mathbb{R}^5$, $\lambda(t) = \left(\lambda_1(t), \lambda_2(t),
\lambda_3(t), \lambda_4(t), \lambda_5(t)\right)$, such that
\begin{equation*}
\dot{S} = \frac{\partial H}{\partial \lambda_1} \, , \quad
\dot{L}_1= \frac{\partial H}{\partial \lambda_2} \, , \quad
\dot{I}= \frac{\partial H}{\partial \lambda_3} \, , \quad
\dot{L}_2 = \frac{\partial H}{\partial \lambda_4} \, , \quad
\dot{R} = \frac{\partial H}{\partial \lambda_5}
\end{equation*}
and
\begin{equation}
\label{adjsystemPMP}
\dot{\lambda}_1 = -\frac{\partial H}{\partial S} \, , \quad
\dot{\lambda}_2 = -\frac{\partial H}{\partial L_1} \, , \quad
\dot{\lambda}_3 = -\frac{\partial H}{\partial I} \, , \quad
\dot{\lambda}_4 = -\frac{\partial H}{\partial L_2} \, , \quad
\dot{\lambda}_5 = -\frac{\partial H}{\partial R} \, .
\end{equation}
The minimality condition
\begin{equation}
\label{maxcondPMP}
\begin{split}
H(S^*(t), &L_1^*(t), I^*(t), L_2^*(t), R^*(t),
\lambda^*(t), u_1^*(t), u_2^*(t))\\
&= \min_{0 \leq u_1, u_2 \leq 1}
H(S^*(t), L_1^*(t), I^*(t), L_2^*(t), R^*(t), \lambda^*(t), u_1, u_2)
\end{split}
\end{equation}
holds almost everywhere on $[0, t_f]$. Moreover, the transversality conditions
\begin{equation*}
\lambda_i(t_f) = 0, \quad
i =1,\ldots, 5 \, ,
\end{equation*}
hold.

\begin{lemma}
\label{lem:thm}
For problem \eqref{modelGab_controls}--\eqref{mincostfunct} with fixed initial conditions
$S(0)$, $L_1(0)$, $I(0)$, $L_2(0)$ and $R(0)$ and fixed final time $t_f$,
there exists adjoint functions $\lambda_1^*(\cdot)$, $\lambda_2^*(\cdot)$,
$\lambda_3^*(\cdot)$, $\lambda_4^*(\cdot)$ and $\lambda_5^*(\cdot)$ such that
\begin{equation}
\label{adjoint_function}
\begin{cases}
\dot{\lambda^*_1}(t) = \lambda^*_1(t) \left(\frac{\beta}{N} I^*(t)
+ \mu \right) - \lambda^*_2(t) \frac{\beta}{N} I^*(t) \\[0.1 cm]
\dot{\lambda^*_2}(t) = \lambda^*_2(t)\left(\delta + \tau_1 + \mu\right)
- \lambda^*_3(t) \phi \delta - \lambda^*_4(t) (1 - \phi) \delta
- \lambda^*_5(t)\tau_1 u^*_1(t) \\[0.1 cm]
\dot{\lambda^*_3}(t) = -W_0 + \lambda^*_1(t) \frac{\beta}{N} S^*(t)
- \lambda^*_2(t) \frac{\beta}{N}(S^*(t) + \sigma L_2^*(t) + \sigma_R R^*(t))\\
\qquad \quad + \lambda^*_3(t) \left(\tau_0 + \mu\right)
+\lambda^*_4(t)\sigma \frac{\beta}{N} L_2^*(t)
- \lambda^*_5(t)\left(\tau_0 - \sigma_R \frac{\beta}{N} R^*(t) \right) \\[0.1 cm]
\dot{\lambda^*_4}(t) = - \lambda^*_2(t) \frac{\beta}{N}I^*(t)
\sigma - \lambda^*_3(t) \omega + \lambda^*_4(t)\left(\sigma \frac{\beta}{N} I^*(t)
+ \omega + \tau_2 u^*_2(t) + \mu\right)\\
\qquad \quad - \lambda^*_5(t)\left( \tau_2 u^*_2(t) \right) \\[0.1 cm]
\dot{\lambda^*_5}(t) = -\lambda^*_2(t) \sigma_R \frac{\beta}{N}I^*(t)
- \lambda^*_3(t) \omega_R + \lambda^*_5(t)\left(\sigma_R \frac{\beta}{N} I^*(t)
+\omega_R + \mu\right) \, ,
\end{cases}
\end{equation}
with transversality conditions
\begin{equation*}
\lambda^*_i(t_f) = 0,
\quad i=1, \ldots, 5 \, .
\end{equation*}
Furthermore,
\begin{equation}
\label{optcontrols}
\begin{split}
u_1^*(t) &= \min \left\{ \max \left\{0, \frac{\tau_1 L_1^*
\left(\lambda^*_2 - \lambda^*_5\right)}{W_1}\right\}, 1 \right\} \, ,\\
u_2^*(t) &= \min \left\{ \max \left\{0, \frac{\tau_2 L^*_2
\left(\lambda^*_4 - \lambda^*_5\right)}{W_2}\right\}, 1 \right\}  \, .
\end{split}
\end{equation}
\end{lemma}

\begin{proof}
System \eqref{adjoint_function} is derived from the Pontryagin maximum principle
(see \eqref{adjsystemPMP}, \cite{Pontryagin_et_all_1962}) and the optimal controls
\eqref{optcontrols} come from the minimality condition \eqref{maxcondPMP}.
For small final time $t_f$, the optimal control pair given by \eqref{optcontrols}
is unique due to the boundedness of the state and adjoint functions and the Lipschitz property
of systems \eqref{modelGab_controls} and \eqref{adjoint_function}
(see \cite{SLenhart_2002} and references cited therein).
\end{proof}

\begin{proof}[Proof of Theorem~\ref{the:thm}]
Existence of an optimal solution $\left(S^*, L_1^*, I^*, L_2^*, R^*\right)$
associated to an optimal control pair $\left(u_1^*, u_2^*\right)$ comes from
the convexity of the integrand of the cost function $\mathcal{J}$ with respect
to the controls $(u_1, u_2)$ and the Lipschitz property of the state system
with respect to state variables $\left(S, L_1, I, L_2, R\right)$
(see, \textrm{e.g.}, \cite{Cesari_1983,Fleming_Rishel_1975}).
For small final time $t_f$, the optimal control pair is given by \eqref{optcontrols}
that is unique by the Lemma above. Because the state system \eqref{modelGab_controls}
is autonomous, uniqueness is valid for any time $t_f$ and not only for small time $t_f$.
\end{proof}


\section{Sensitivity analysis to the duration of intervention $t_f$}
\label{app:tf:B}

We fix $\beta=100$ and $\sigma_R=\sigma$ and the remaining parameters according
to Table~\ref{parameters} and vary $t_f$. Results for the proportion of infectious
individuals are shown in the Figure~\ref{fig:I:beta100:sigual:tf}.
\begin{figure}[!ht]
\centering
\includegraphics[scale=0.70]{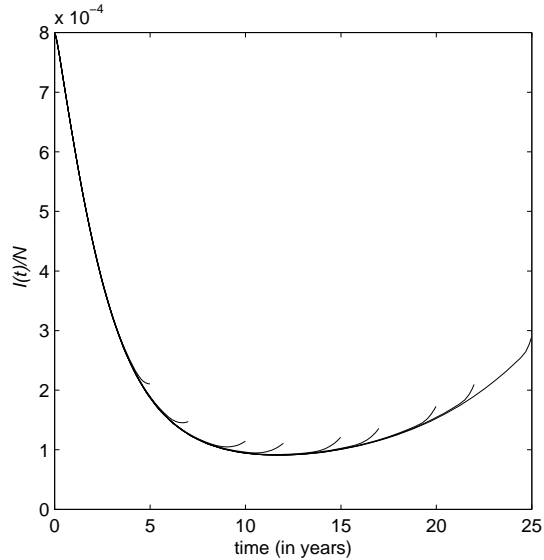}
\caption{Proportion of infectious individuals for the optimal solution $I(t)$
with $t_f \in \{5, 7, 10, 12, 15, 17, 20, 22, 25 \}$. Parameters according
to Table~\ref{parameters}, $\beta=100$ and $\sigma_R=\sigma$.}
\label{fig:I:beta100:sigual:tf}
\end{figure}
The general behaviour do not change significantly with $t_f$. The proportion
of infected individuals slightly increases towards the end of the intervention
for $t_f>7$. This tendency is more pronounced for higher $t_f$.


\section{Sensitivity analysis to the weight constants on the objective functional $\mathcal{J}$}
\label{app:Wis:C}

Figure~\ref{fig:secanal:Wi} shows the results for different combination
of the weight constants on the objective functional $\mathcal{J}$. We fix $\beta=100$
and $\sigma_R=\sigma$ and the remaining parameters according to Table~\ref{parameters}
and vary $W_0$, $W_1$ and $W_2$. Efficacy decreases when the costs $W_1$ and $W_2$ increase,
corresponding to an earlier relaxation of the intensity of treatment $(u_1(t), u_2(t))$
in the optimal solution due to cost restrictions.
\begin{figure}[!ht]
\centering
\subfloat[\footnotesize{$W_0=50$ and varying $W_1=W_2$.}]{\label{fig:secanal:W12}
\includegraphics[scale=0.56]{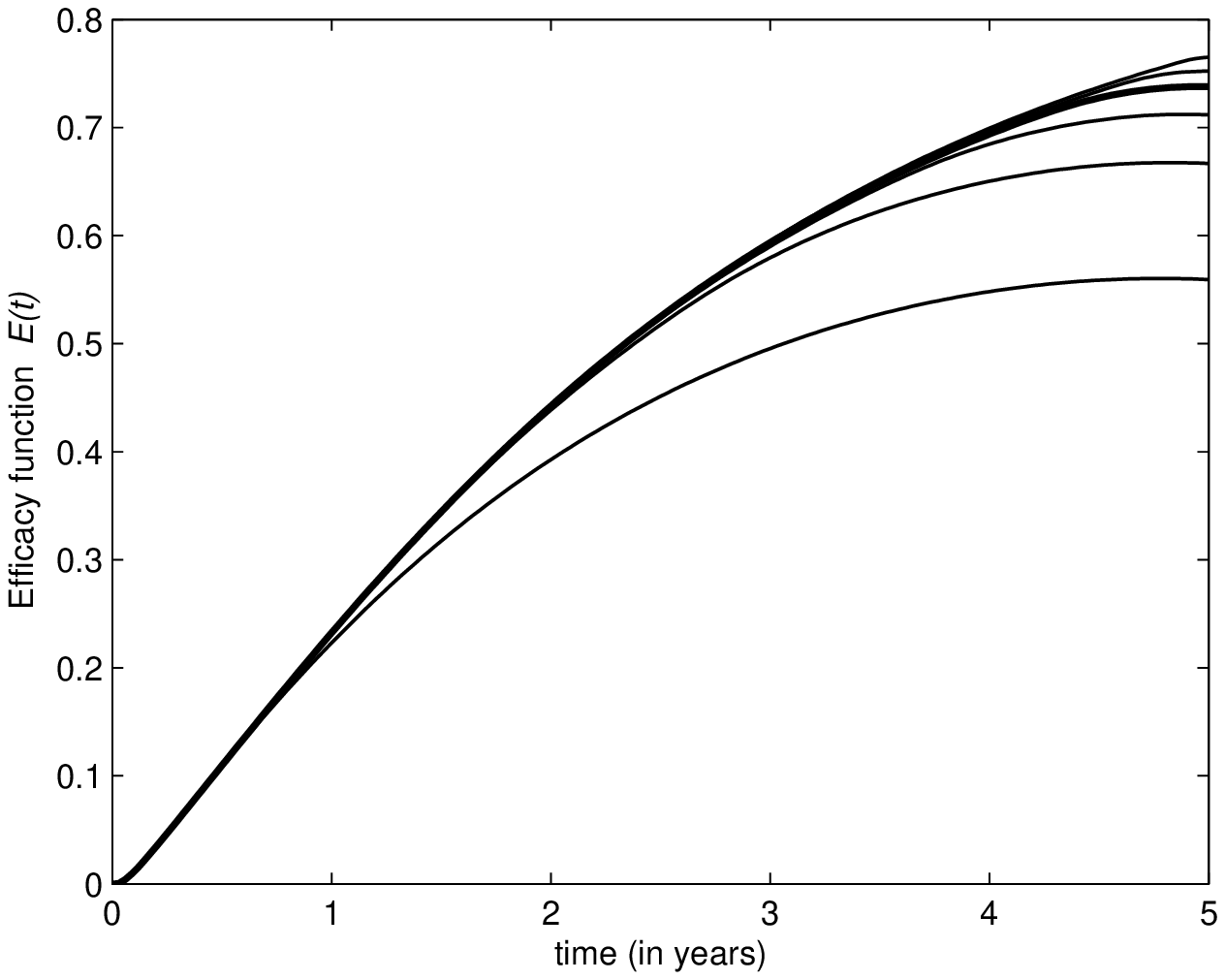}}
\subfloat[\footnotesize{$W_1=W_2=50$ and varying $W_0$.}]{\label{fig:secanal:W0}
\includegraphics[scale=0.60]{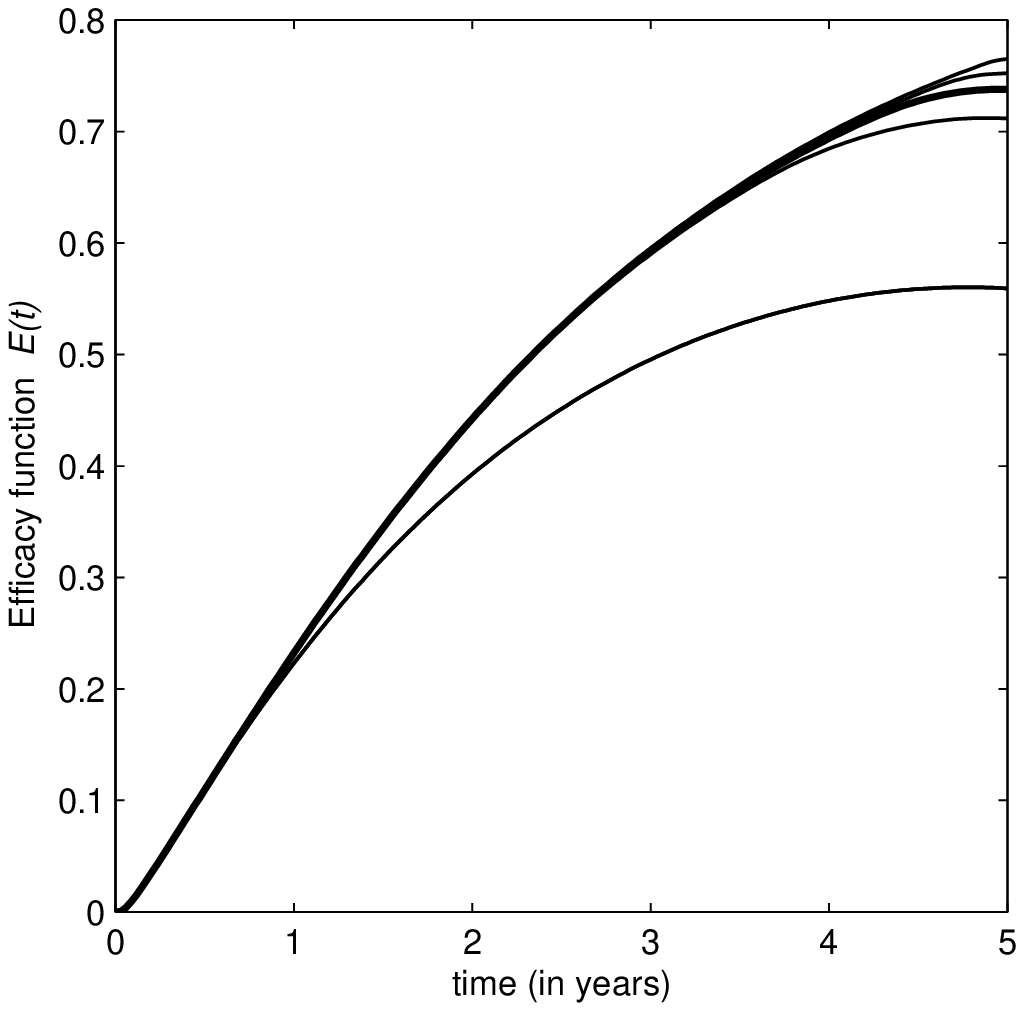}}\\
\subfloat[\footnotesize{$W_0=W_2=50$ and varying $W_1$.}]
{\label{fig:secanal:W1}
\includegraphics[scale=0.58]{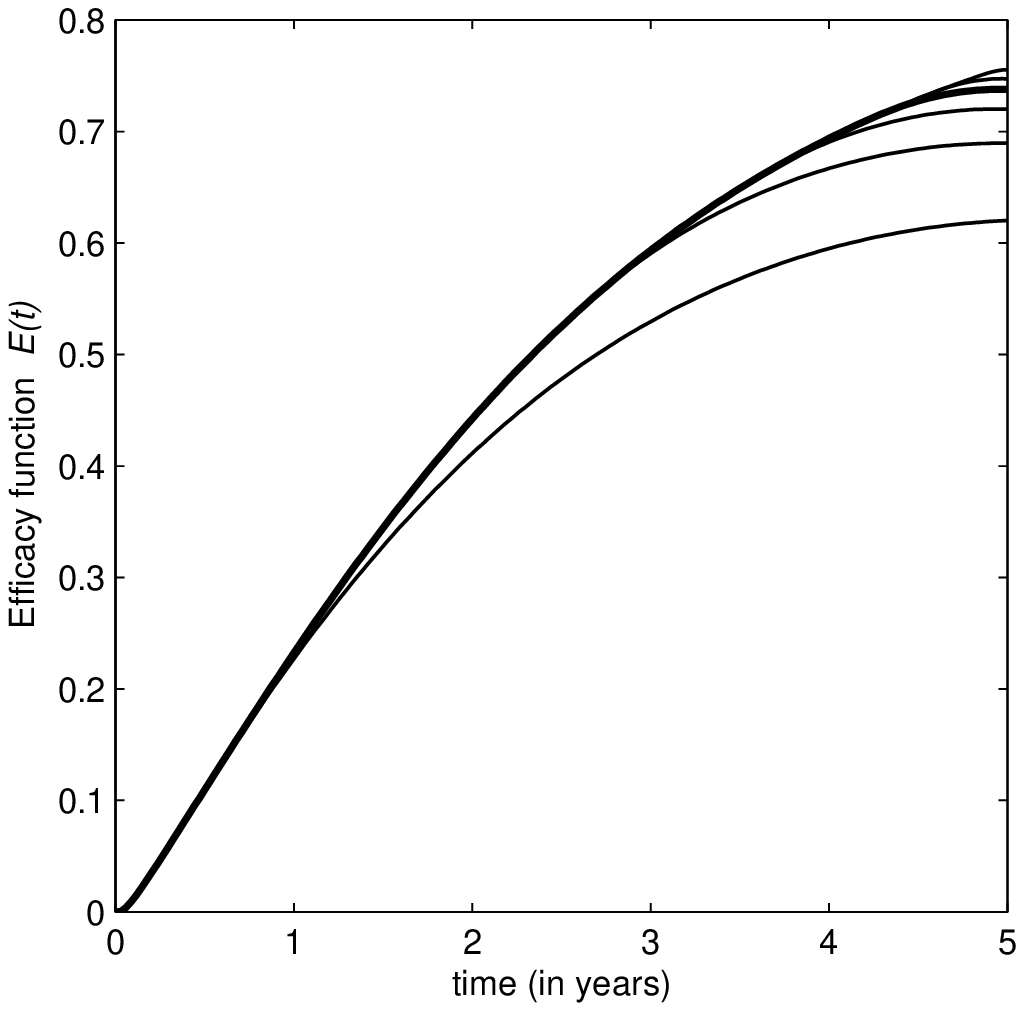}}
\subfloat[\footnotesize{$W_0=W_1=50$ and varying $W_2$.}]
{\label{fig:secanal:W2}
\includegraphics[scale=0.58]{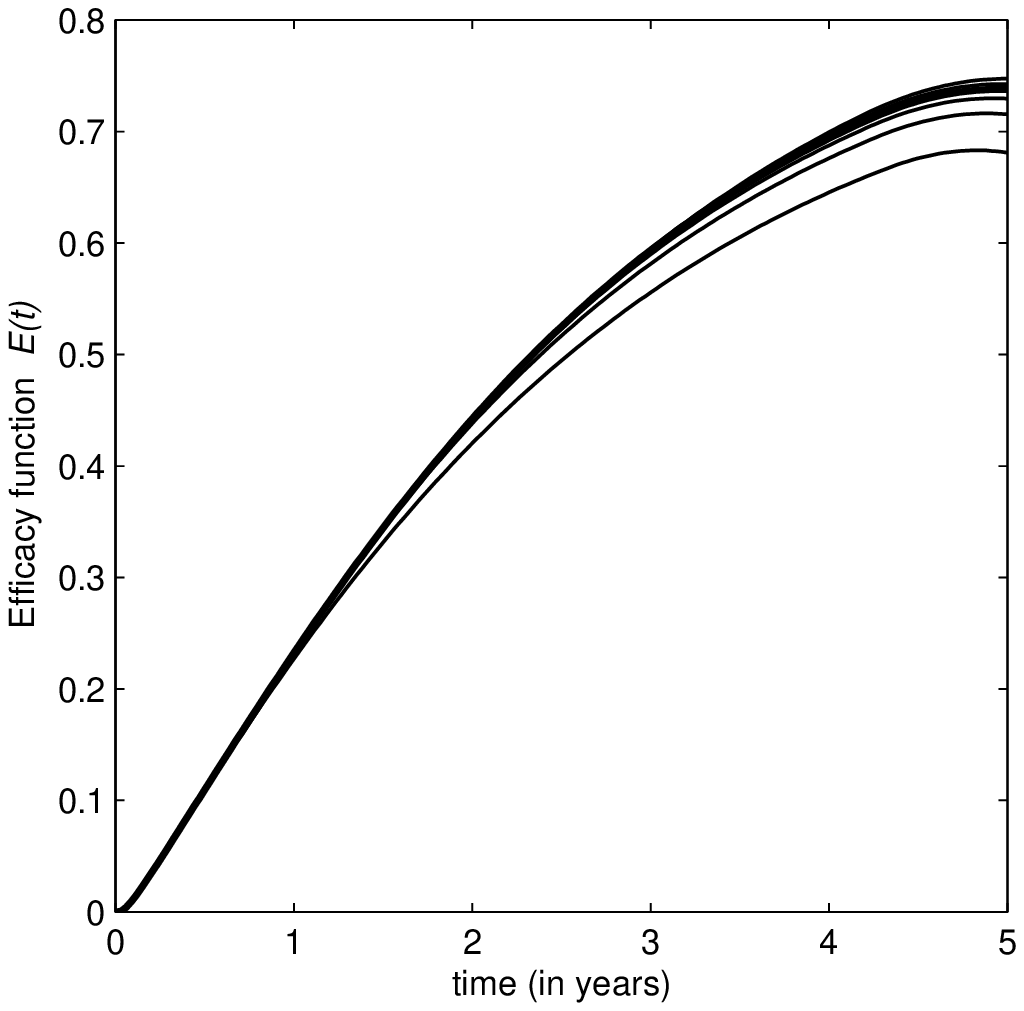}}
\caption{Sensitivity analysis to the weight constants on the objective functional $\mathcal{J}$.
{\bf (a)} $W_0=50$ and $W_1=W_2=$5, 25, 50, 100, 200, 500. {\bf (b)} $W_1=W_2=50$ and
$W_0=$5, 25, 50, 100, 200, 500. {\bf (c)} $W_0=W_2=50$ and $W_1=$5, 25, 50, 100, 200, 500.
{\bf (d)} $W_0=W_1=50$ and $W_2=$5, 25, 50, 100, 200, 500.}
\label{fig:secanal:Wi}
\end{figure}
The change in efficacy is more pronounced for the cases where the weight associated
with infectious individuals $W_0$ change in comparison to the weights associated
with the controls $W_1=W_2$ (Figures~\ref{fig:secanal:W12} and \ref{fig:secanal:W0}).
Results are less sensitive to the variation between the weight controls $W_1$ and $W_2$
(Figures~\ref{fig:secanal:W1} and \ref{fig:secanal:W2}).


\section*{Acknowledgments}

This work was partially supported by the
Portuguese Foundation for Science and Technology (FCT)
through the: \emph{Centro de Matem\'{a}tica e Aplica\c{c}\~{o}es},
project PEst-OE/MAT/UI0297/2014 (Rodrigues);
\emph{Center for Research and Development in Mathematics and Applications} (CIDMA),
project PEst-OE/MAT/UI4106/2014 (Silva and Torres);
post-doc fellowship SFRH/BPD/72061/2010 (Silva);
project PTDC/EEI-AUT/1450/2012, co-financed by FEDER
under POFC-QREN with COMPETE reference FCOMP-01-0124-FEDER-028894 (Torres).



\end{document}